\documentclass[10pt,a4paper]{article}
\usepackage[utf8]{inputenc}
\usepackage[T1]{fontenc}
\usepackage{amsmath}
\usepackage{amsfonts}
\usepackage{amssymb}
\usepackage{pgf, tikz}
\usepackage{subcaption}
\usepackage{float}
\usepackage{makeidx}
\usepackage{graphicx}
\usepackage{dsfont}
\usepackage{color}
\usepackage{url}
\usepackage{amsthm}
\usepackage{soul}
\usepackage{multicol}
\usepackage{hyperref}
\hypersetup{
	colorlinks=true,
	linkcolor=blue,
	filecolor=blue,      
	urlcolor=blue,
	citecolor=blue
}

\usepackage[margin=1in]{geometry}

\author{Leo Hahn}
\title{Steady state and mixing of two run-and-tumble particles interacting through jamming and attractive forces}
\date{{\it Laboratoire de Mathématiques Blaise Pascal UMR 6620, CNRS, \\ Université Clermont Auvergne, Aubière, France} \\[\baselineskip] \today}

\newcommand{\tv}[1]{{\left\Vert #1 \right\Vert_{\text{\normalfont TV}}}}

\newtheorem{Thm}{Theorem}
\newtheorem{Prop}[Thm]{Proposition}
\newtheorem{Lem}[Thm]{Lemma}
\newtheorem{Def}[Thm]{Definition}
\newtheorem{Not}[Thm]{Notation}

\newtheorem{Rem}[Thm]{Remark}

\begin{document}

\maketitle

\abstract{We study the long-time behavior of two run-and-tumble particles on the real line subjected to an attractive interaction potential and jamming interactions, which prevent the particles from crossing. We provide the explicit invariant measure, a useful tool for studying clustering phenomena in out-of-equilibrium statistical mechanics, for different tumbling mechanisms and potentials. An important difference with invariant measures of equilibrium systems are Dirac masses on the boundary of the state space, due to the jamming interactions. Qualitative changes in the invariant measure depending on model parameters are also observed, suggesting, like a growing body of evidence, that run-and-tumble particle systems can be classified into close-to-equilibrium and strongly out-of-equilibrium models. We also study the relaxation properties of the system, which are linked to the timescale at which clustering emerges from an arbitrary initial configuration. When the interaction potential is linear, we show that the total variation distance to the invariant measure decays exponentially and provide sharp bounds on the decay rate. When the interaction potential is harmonic, we give quantitative exponential bounds in a Wasserstein-type distance.}

\section{Introduction}

Run-and-tumble particles (RTPs), used to model bacteria such as E.~coli~\cite{berg72,schnitzer93,berg04}, are characterized by random piecewise linear motion. In order to self-propel, RTPs consume energy at the particle level making them a prime example of active particles. This leads to out-of-equilibrium phenomena such as clustering~\cite{cates15} and accumulation at boundaries~\cite{elgeti15}, which remain active areas of research~\cite{smith22,metson23,arnoulx23}. In fact, these two features are strongly linked since cluster configurations are boundary configurations in joint position space~\cite{slowman16,ledoussal21,hahn23}. The two phenomena can be investigated by studying which configurations are typical in the sense that they have increased mass under the invariant measure~\cite{slowman17,angelani17,hahn23}. And, as an important second step, by determining the speed of convergence towards the invariant measure to understand at which timescale it becomes relevant~\cite{mallmin19,das20,guillin24}.

Accumulation at boundaries is already visible on the steady state of a single one-dimensional run-and-tumble particle. Indeed, the steady state of a RTP subjected to thermal noise and confined in a bounded interval differs from the Boltzmann measure and displays exponential accumulation at the boundaries~\cite{malakar18}. In the limit of vanishing thermal noise, Dirac masses appear on the boundary, showing an even stronger form of accumulation. The same phenomenon can also be observed on a one-dimensional RTP subjected to a confining potential. In this case~\cite{dhar19,basu20}, the support of the steady state is a bounded interval \( [x_-, x_+] \), which stands in stark contrast to the Boltzmann setting. Furthermore, depending on model parameters two very different behaviors can be observed at the boundaries \( x_\pm \). There are the close-to-passive configurations where the density of the steady state vanishes at the boundary and strongly out-of-equilibrium configurations where it diverges, indicating accumulation. Clustering, on the other hand, is more challenging as it requires the study of multiple interacting active particles. Such systems can be simulated~\cite{dolai20, turci21, deblais18} but exact results remain limited. The two particle case constitutes a notable exception because it can be reduced to the one particle case by considering the relative particle. This comes at the price of dealing with a more involved dynamics for the velocity. A first example of clustering, which can again be studied through its steady state, is given by two RTPs interacting through an attractive potential~\cite{ledoussal21}. Considering the relative particle maps two particles interacting through the potential {\( V(x) = \mu x^2/2 \)} to~\cite{basu20}. Hence the support of the relative particle's steady state is a bounded interval \( [x_-, x_+] \) and all hallmarks of accumulation at \( x_\pm \) can be reinterpreted as clustering. Clustering is also present in the steady state of two RTPs on a discrete 1D torus interacting through jamming both when particles have two~\cite{slowman16} and three internal states~\cite{slowman17}. In fact, generalizing to arbitrary tumbling mechanisms~\cite{hahn23} reveals the existence of two distinct universality classes. In the detailed jamming class, the invariant measure has Dirac masses on the boundary, indicating clustering, as well as uniform terms. In the global jamming class there are additional exponential terms. Hence the detailed jamming class is closer to the equilibrium setting, where there are only uniform terms, than the global jamming class.

This paper is dedicated to three processes modeling two RTPs on the real line subjected to an attractive interaction potential as in~\cite{ledoussal21}, with the difference of added jamming interactions, which prevent the particles from crossing. We consider the case where each particle velocity takes the values \( \pm 1 \) and the interaction potential is \( V(x) = c|x| \) (instantaneous linear process) or \( V(x) = \mu x^2/2 \) (instantaneous harmonic process). When \( V(x) = c|x| \), unlike~\cite{ledoussal21}, we also consider the case where the particle velocity can take the additional value \( 0 \) to account for particle reorientation, which is not instantaneous. This tumble mechanism is closer to the movement of actual bacteria such as E.~coli~\cite{slowman17,saragosti12}. For continuous-space processes with jamming, the boundary conditions of the invariant measure's Fokker-Planck equation are a persistent challenge. Rather than address it directly, different strategies have been developed to circumvent the problem. Examples include considering increasingly fine space discretizations~\cite{slowman16}, more and more peaked soft potentials~\cite{arnoulx19} and vanishing thermal noise~\cite{malakar18}. The shared idea being to introduce an approximate process with simpler boundary behavior, thus allowing the computation of the approximate invariant measure, and then taking a limit to recover the original invariant measure. The first drawback of these methods is that the approximate invariant measure is more complex than the original invariant measure. The second is that a limit has to be computed to recover the original invariant measure. Continuing the approach in~\cite{hahn23}, we use the mathematical framework of piecewise deterministic Markov processes (PDMPs)~\cite{davis93,davis84}, specifically their explicit generator, to work directly with the original process and avoid the drawbacks of approximation. With this technique, the boundary conditions satisfied by the invariant measure are implicit in the domain of the generator. All three processes show strong marks of clustering in the form of Dirac masses on the boundary and the instantaneous and finite linear process also display exponential concentration of mass. The density of the instantaneous harmonic process can diverge at the boundary for certain parameter choices. Furthermore, the behavior of the steady state of the finite linear process is particularly rich as its form depends on the model parameters. When the ratio of the velocity of the particles and the strength of the attractive interaction is below a certain threshold, the steady state is a product measure outside of the boundary. This product form is lost and one of the Dirac masses on the boundary disappears when the threshold is crossed. This is reminiscent of the separation into close-to-equilibrium and strongly out-of-equilibrium universality classes in~\cite{hahn23} and the shape transition in~\cite{dhar19,basu20}.

Finally, we tackle the question of the speed of convergence of these processes toward their steady state.\@ Indeed, while their steady state determines their long time behavior, it is irrelevant at small time scales. Understanding precisely when the asymptotic behavior takes over is thus of paramount importance. We depart from the spectral gap approach in~\cite{mallmin19,malakar18} which provides asymptotic results in the form of exponential decay with an uncontrolled prefactor. Instead, we obtain non-asymptotic bounds by constructing explicit couplings (see~\cite{malrieu16} and references therein for a review of convergence results for PDMPs). We start by showing exponential convergence in total variation of the instantaneous linear and finite linear process. The first key idea is to use the same synchronous coupling as in~\cite{guillin24} ensuring a weak form of order preservation that essentially reduces the study of convergence to the study of a hitting time. The second key idea, and the main difference with~\cite{guillin24}, is to use non-asymptotic large deviation results~\cite{wu2000deviation,lezaud01,cattiaux08} to bound the Laplace transform of this hitting time. We also show the sharpness of our bounds, a less studied but important aspect, by identifying obstacles to mixing as in~\cite{guillin24}. The method, which bears a loose resemblance to~\cite{fontbona12,lund96}, is general and can be extended to arbitrary tumbling mechanisms and linear potentials. In the case of the instantaneous harmonic process, we stick to the coupling method but work in Wasserstein instead of total variation distance to take advantage of the contractivity of the underlying deterministic dynamics. The proof is adapted from the coupling argument in~\cite{benaim12}, which was extended in~\cite{cloez15}.

The rest of the article is organized as follows. The remainder of this section gives a detailed description of the mathematical model and summarizes the main results. In Section~\ref{sec:invariant_measures} explicit formulae for the steady state are established. Finally, Section~\ref{sec:convergence} establishes quantitative convergence rates and Section~\ref{sec:optimality} shows that these rates are in many cases optimal.

\subsection{Model}

We consider two point particles on the real line interacting through an attractive potential. Their motion is also subjected to telegraphic noise and jamming interactions. Using the positions $x_1, x_2$ and the velocities $\sigma_1, \sigma_2$ of the particles this can be modeled as follows
\begin{itemize}
\item the $\sigma_i$ are independent Markov processes with the transition rates of Figure~\ref{fig:instantaneous_tumble_single_particle_velocity_transition_rates} (resp.~\ref{fig:finite_tumble_single_particle_velocity_transition_rates}),
\item the positions $x_i$ evolve according to the ODEs
$$
\partial_t x_1 = -V'(x_1 - x_2) + v \sigma_1(t), \quad \partial_t x_2 = -V'(x_2 - x_1) + v \sigma_2(t),
$$
where $V$ is an interaction potential satisfying $V(-x) = V(x)$.
\end{itemize}

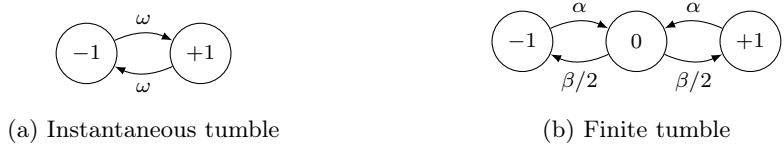
\begin{figure}[H]
	\centering
	\begin{subfigure}[t]{0.4\textwidth}
		\centering
		\begin{tikzpicture}[-latex, node distance=1.5cm, main/.style = {draw, circle, minimum size=.8cm, font=\footnotesize}]
		\node[main] (I+) {$-1$};
		\node[main] (I-) [right of=I+] {$+1$};
		
		\path[every node/.style={font=\footnotesize}]
		(I+) edge[bend left=25] node [above] {$\omega$} (I-)
		(I-) edge[bend left=25] node [below] {$\omega$} (I+)
		;
		\end{tikzpicture}
		\caption{Instantaneous tumble}
		\label{fig:instantaneous_tumble_single_particle_velocity_transition_rates}
	\end{subfigure}
	\begin{subfigure}[t]{0.4\textwidth}
		\centering
		\begin{tikzpicture}[-latex, node distance=1.5cm, main/.style = {draw, circle, minimum size=.8cm, font=\footnotesize}]
		\node[main] (F+) [right of=I-] {$-1$};
		\node[main] (F0) [right of=F+] {$0$};
		\node[main] (F-) [right of=F0] {$+1$};
		\path[every node/.style={font=\footnotesize}]
		(F+) edge[bend left=25] node [above] {$\alpha$} (F0)
		(F0) edge[bend left=25] node [below] {$\beta/2$} (F+)
		(F0) edge[bend right=25] node [below] {$\beta/2$} (F-)
		(F-) edge[bend right=25] node [above] {$\alpha$} (F0)
		;
		\end{tikzpicture}
		\caption{Finite tumble}
		\label{fig:finite_tumble_single_particle_velocity_transition_rates}
	\end{subfigure}
	\caption{Single-particle velocity transition rates}
	\label{fig:single_particle_velocity_transition_rates}
\end{figure}

The positions do not reach a steady state. Thus we focus on the relative position $x = x_2 - x_1$ and the relative velocity $\sigma = \sigma_2 - \sigma_1$. The relative position is governed by the ODE
$$
\partial_t x = -2 V'(x) + v \sigma(t)
$$
when $x>0$ and by jamming interactions, which will be detailed later, when $x=0$. The relative velocity follows the transition rates \ref{fig:instantaneous_tumble_relative_particle_velocity_transition_rates} (resp.~\ref{fig:finite_tumble_relative_particle_velocity_transition_rates}). We denote $\Sigma = \{-2, 0, 2\}$ (resp.~$\Sigma = \{-2, -1, 0_\pm, 0_0, 1, 2\}$) the state space of $\sigma$ and adopt the convention $v \cdot 0_\pm = v \cdot 0_0 = 0$.

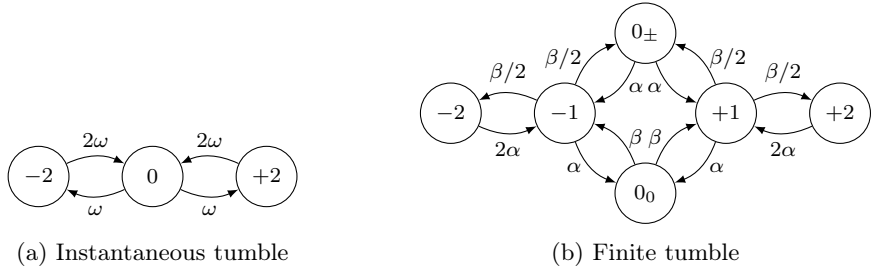
\begin{figure}[H]
	\centering
	\begin{subfigure}[t]{0.4\textwidth}
		\centering
		\begin{tikzpicture}[-latex, node distance=1.5cm, main/.style = {draw, circle, minimum size=.8cm, font=\footnotesize}]
		\node[main] (+2) {$-2$};
		\node[main] (0) [right of=+2] {$0$};
		\node[main] (-2) [right of=0] {$+2$};
		\path[every node/.style={font=\footnotesize}]
		(0) edge[bend right=25] node [below] {$\omega$} (-2)
		(0) edge[bend left=25] node [below] {$\omega$} (+2)
		(-2) edge[bend right=25] node [above] {$2\omega$} (0)
		(+2) edge[bend left=25] node [above] {$2\omega$} (0);
		\end{tikzpicture}
		\caption{Instantaneous tumble}
		\label{fig:instantaneous_tumble_relative_particle_velocity_transition_rates}
	\end{subfigure}
	\begin{subfigure}[t]{0.4\textwidth}
		\centering
		\begin{tikzpicture}[-latex, node distance=1.5cm, main/.style = {draw, circle, minimum size=.8cm, font=\footnotesize}]
		\node[main] (2) {$-2$};
		\node[main] (1) [right of=2] {$-1$};
		\node[main] (0pm) [above right  of=1] {$0_\pm$};
		\node[main] (00) [below right  of=1] {$0_0$};
		\node[main] (-1) [below right  of=0pm]{$+1$};
		\node[main] (-2) [right of=-1]{$+2$};
		\path[every node/.style={font=\footnotesize}]
		(2) edge[bend right=25] node [below] {$2\alpha$} (1)
		(-2) edge[bend left=25] node [below] {$2\alpha$} (-1)
		(1) edge[bend right=25] node [above] {$\beta/2$} (2)
		(1) edge[bend left=25] node [left] {$\beta/2$} (0pm)
		(1) edge[bend right=25] node [left] {$\alpha$} (00)
		(-1) edge[bend left=25] node [above] {$\beta/2$} (-2)
		(-1) edge[bend right=25] node [right] {$\beta/2$} (0pm)
		(-1) edge[bend left=25] node [right] {$\alpha$} (00)
		(0pm) edge[bend left=25] node [right] {$\alpha$} (1)
		(0pm) edge[bend right=25] node [left] {$\alpha$} (-1)
		(00) edge[bend right=25] node [right] {$\beta$} (1)
		(00) edge[bend left=25] node [left] {$\beta$} (-1)	
		;
		\end{tikzpicture}
		\caption{Finite tumble}
		\label{fig:finite_tumble_relative_particle_velocity_transition_rates}
	\end{subfigure}
	\caption{Relative velocity transition rates}
	\label{fig:relative_particle_velocity_transition_rates}
\end{figure}

Note that if the $\sigma_i$ follow figure \ref{fig:finite_tumble_single_particle_velocity_transition_rates} then $\sigma_2 - \sigma_1$ is not Markovian. Therefore, we split the state $\sigma_2 - \sigma_1 = 0$ into the states $0_\pm$ (corresponding to $\sigma_1 = \sigma_2 = \pm 1$) and $0_0$ (corresponding to $\sigma_1=\sigma_2=0$) to recover a Markov jump process (see figure \ref{fig:finite_tumble_relative_particle_velocity_transition_rates}). \\

Assume without loss of generality that $x(0) = x_2(0) - x_1(0) \ge 0$ and define the behavior of the process when the particles collide, i.e.~when $x(t) = 0$, as follows
\begin{itemize}
\item if $-2 V'(0) + v \sigma(t) > 0$ the self-propulsion overcomes the attraction resulting from the interaction potential and pushes the particles apart so the relative position $x$ follows the ODE $\partial_t x = -2 V'(x) + v \sigma(t)$ and immediately becomes positive,
\item if $-2 V'(0) + v \sigma(t) \le 0$ the attractive forces overcome the self-propulsion so the particles are `glued' together and $x$ remains $0$ until a velocity change occurs.
\end{itemize}
These jamming interactions lead to $x(t) \ge 0$ for all $t \ge 0$ and hence to the absence of particle crossings. The process can thus be recursively defined as follows.
\begin{Def}[Interacting run-and-tumble process] \label{def:generic_process} Let $\sigma(t)$ be a Markov jump process with the transition rates \ref{fig:instantaneous_tumble_relative_particle_velocity_transition_rates} (resp.~\ref{fig:finite_tumble_relative_particle_velocity_transition_rates}) and recursively construct
	$$
	x(t) = \max\left[0, \phi^{\sigma(T_n+)}_{t - T_n} \left(x(T_n)\right)\right] \text{ for } t \in [T_n, T_{n+1})
	$$
	where $0 = T_0 < T_1 < \cdots$ are the jump times of $\sigma(t)$ and $(\phi^\sigma_t)_{t \ge 0}$ is the flow associated to the ODE \mbox{$\partial_t x = -2 V'(x) + v \sigma$}. We call the $(\mathbb R_+ \times \Sigma)$-valued process $X(t) = (x(t), \sigma(t))$ interacting run-and-tumble process.
\end{Def}

This paper is dedicated to the long-time behavior of three special cases of the interacting run-and-tumble process
\begin{itemize}
\item the \textit{instantaneous linear} process where the particles are assumed to have {\it instantaneous} tumbles, i.e.~$\sigma$ follows figure \ref{fig:instantaneous_tumble_relative_particle_velocity_transition_rates}, and to interact through the {\it linear} potential $V(x) = c|x|$ (see figure \ref{fig:instantaneous_linear_realization}),
\item the \textit{finite linear} process where the particles are assumed to have {\it finite} tumble duration, i.e.~$\sigma$ follows figure \ref{fig:finite_tumble_relative_particle_velocity_transition_rates}, and to interact through the {\it linear} potential $V(x) = c|x|$ (see figure \ref{fig:finite_linear_realization}),
\item the \textit{instantaneous harmonic} process where the particles are assumed to have {\it instantaneous} tumbles, i.e.~$\sigma$ follows figure \ref{fig:instantaneous_tumble_relative_particle_velocity_transition_rates}, and to interact through the {\it harmonic} potential $V(x) = \frac{\mu}{2}x^2$ (see Figure~\ref{fig:instantaneous_harmonic_realization}).
\end{itemize}

Note that the linear potential (which resembles the 1D Coulomb interaction~\cite{ledoussal21}) and the harmonic potential are part of the larger class \( V(x) = a |x|^p \) of physical interactions also studied in~\cite{dhar19,gueneau23}. It would be interesting to consider the \textit{finite harmonic} process (with the interaction potential \( V(x) = \frac\mu2x^2 \) and the rates of Figure~\ref{fig:finite_tumble_relative_particle_velocity_transition_rates}), but computing its invariant measure explicitly requires solving a system of linear differential equations that seems intractable because of its size and non-constant coefficients (see Section~\ref{sec:instantaneous_linear_process}).

\begin{figure}[H]
	\centering
	\begin{subfigure}{0.32\textwidth}
	\centering
	\includegraphics[width=\textwidth]{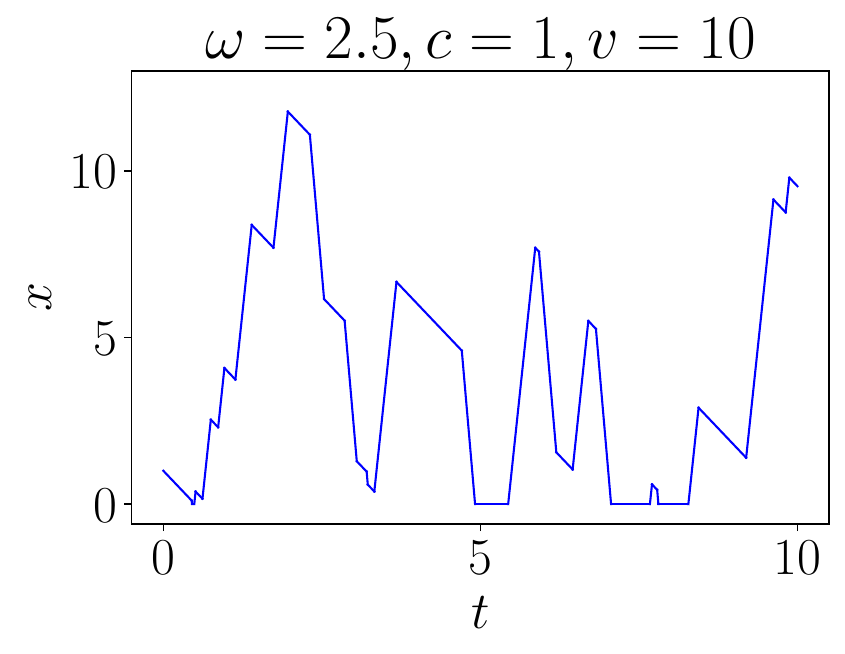}
	\caption{Instantaneous linear process}
	\label{fig:instantaneous_linear_realization}
	\end{subfigure}
	\begin{subfigure}{0.32\textwidth}
	\centering
	\includegraphics[width=\textwidth]{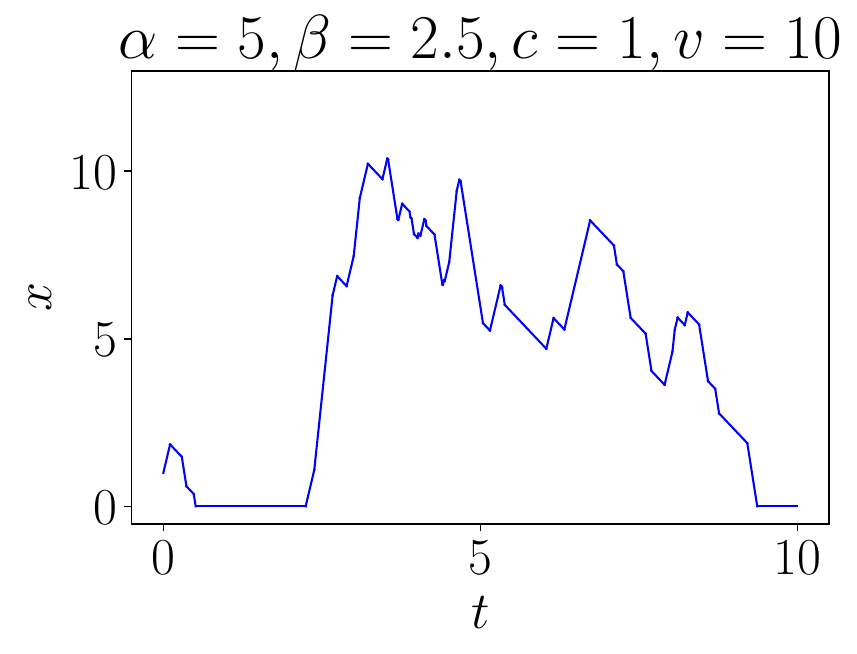}
	\caption{Finite linear process}
	\label{fig:finite_linear_realization}
\end{subfigure}
	\begin{subfigure}{0.32\textwidth}
	\centering
	\includegraphics[width=\textwidth]{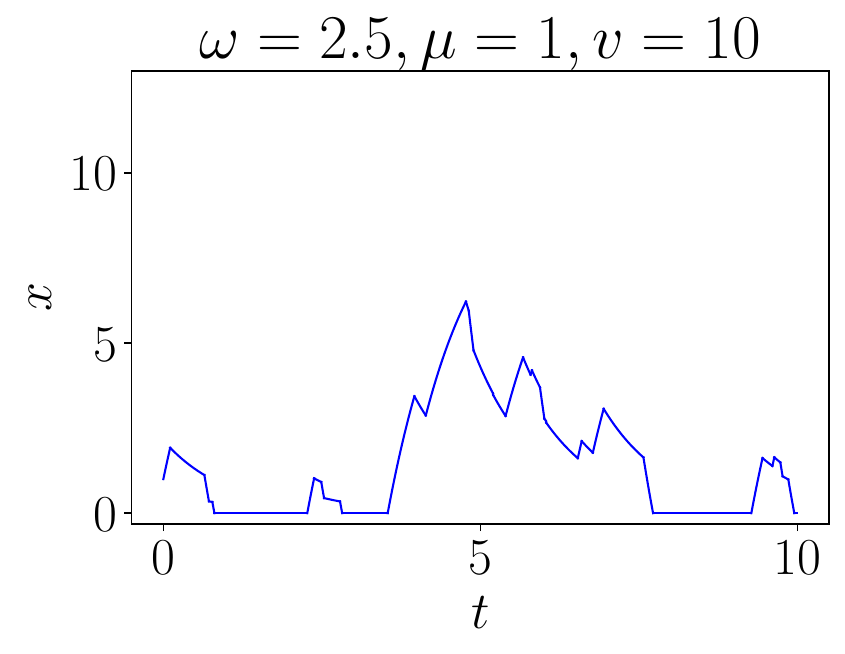}
	\caption{Instantaneous harmonic process}
	\label{fig:instantaneous_harmonic_realization}
\end{subfigure}
	\caption{Realization of the relative distance $x$}
	\label{fig:realizations}
\end{figure}

The following notation is convenient when writing down the extended generator of the interacting run-and-tumble process.

\begin{Not}
Denote $\mathcal Q = (q_{\sigma, \tilde \sigma})_{\sigma, \tilde \sigma \in \Sigma}$ the transition-rate matrix of the $\Sigma$-valued Markov jump process~$\sigma(t)$. In the instantaneous case (Figure~\ref{fig:instantaneous_tumble_relative_particle_velocity_transition_rates})
\begin{equation*}%\label{eq:instantaneous_Q}
\mathcal Q =  \bordermatrix{ & \tilde \sigma = 2 & \tilde \sigma =  0 & \tilde \sigma =  -2 \cr
\sigma = 2 & -2  {\omega} & 2  {\omega} & 0 \cr
\sigma = 0 &{\omega} & -2  {\omega} & {\omega} \cr
\sigma = -2 & 0 & 2  {\omega} & -2  {\omega}},
\end{equation*}
while in the finite case (Figure~\ref{fig:finite_tumble_relative_particle_velocity_transition_rates})
\begin{equation}\label{eq:finite_Q}
\mathcal Q = \bordermatrix{ & \tilde \sigma = 2 & \tilde \sigma = 1 & \tilde \sigma = 0_\pm & \tilde \sigma = 0_0 & \tilde \sigma = -1 & \tilde \sigma = -2 \cr
\sigma = 2 & -2  {\alpha} & 2  {\alpha} & 0 & 0 & 0 & 0 \cr
\sigma =1 & \frac{1}{2}  {\beta} & -{\alpha} - {\beta} & \frac{1}{2}  {\beta} & {\alpha} & 0 & 0 \cr
\sigma = 0_\pm& 0 & {\alpha} & -2  {\alpha} & 0 & {\alpha} & 0 \cr
\sigma = 0_0 & 0 & {\beta} & 0 & -2  {\beta} & {\beta} & 0 \cr
\sigma = -1 & 0 & 0 & \frac{1}{2}  {\beta} & {\alpha} & -{\alpha} - {\beta} & \frac{1}{2}  {\beta} \cr
\sigma = -2 & 0 & 0 & 0 & 0 & 2  {\alpha} & -2  {\alpha} 
}.
\end{equation}
\end{Not}

The interacting run-and-tumble process can be constructed as in~\cite[Section 2]{guillin24}, see also~\cite[Section 24]{davis93} and~\cite[Appendix A]{bierkens23}, leading to the following proposition.

\begin{Prop} \begin{itemize}\label{prop:feller+strong_markov+generator}
		\item[(i)] The interacting run-and-tumble process is a Feller process with the strong Markov property. Its state space is \( E = \mathbb R_+ \times \Sigma \).
		\item[(ii)] A bounded measurable \( f : E \rightarrow \mathbb R \) is in the domain the extended generator \( \mathcal L \) if and only if
		$$
		t \mapsto f(\max(0, \phi^\sigma_t(x)), \sigma) \text{ is absolutely continuous for all } (x, \sigma) \in E.
		$$
		For such \( f \) one has
		$$
		\mathcal L f(x, \sigma) = 1_{\{x > 0 \text{ or } -2V'(0) + v \sigma \ge 0\}} \left( -2V'(x) + v \sigma \right) \partial_x f(x, \sigma) + \sum_{\tilde \sigma} q_{\sigma, \tilde \sigma} f(x, \tilde \sigma).
		$$
	\end{itemize}
\end{Prop}

Two RTPs on the real line interacting through an attractive potential were previously considered in~\cite{ledoussal21}. The models in the present article set themselves apart by the presence of jamming interactions, which prevent the particles from crossing. The challenges coming from the state space's jamming boundary at \( x = 0 \) are handled using the formalism of piecewise-deterministic Markov processes~\cite{hahn23}. This avoids approximations such as space discretization~\cite{slowman16,metson22}, soft potentials~\cite{arnoulx19,arnoulx23} or thermal noise~\cite{das20}. Another important difference with~\cite{ledoussal21} is the richer tumbling mechanism of the finite linear process, which accounts for the non-instantaneous nature of particle reorientation. This more closely models the biological reality of bacteria such as E.~coli~\cite{slowman17,saragosti12}.

\subsection{Main results}

This paper is devoted to the study of the long-time behavior of the instantaneous linear, finite linear and instantaneous harmonic process. Our main results are explicit formulae for their invariant measure and quantitative bounds for the speed of convergence towards these measures. \\

The PDMP formalism~\cite{davis93} provides access to the generator \(\mathcal{L}\) and enables the use of the characterization  
\[
\pi \text{ invariant } \iff \int \mathcal L f d\pi = 0 \text{ for all } f \in D(\mathcal L). 
\]  
This dual formulation of Fokker-Planck facilitates the complete determination of the invariant measure for the three processes considered in this article.

\begin{Thm}
\begin{itemize}
	\item[(i)] The unique invariant measure of the instantaneous linear process is
	$$
	\pi = \sum_{\sigma  \in \Sigma } \left( d_{\sigma } \delta_0 + a_\sigma e^{\zeta x} dx \right) \otimes \delta_{\sigma}
	$$
	where $\zeta, d_\sigma$ and $a_\sigma$ are explicit functions of $\omega, c$ and $v$ (see Proposition~\ref{prop:instantaneous_linear_invariant_measure}). Furthermore
	\[
	d_{+2} = 0 \text{ while } d_{0}, d_{-2} > 0.
	\]
	\item[(ii)] The unique invariant measure of the finite linear process is
	$$
	\pi = \left\{\begin{array}{cl}
	\sum_{\sigma  \in \Sigma} \left( d_\sigma  \delta_0 + c_2 a_{\sigma }\left(\zeta_2\right) e^{\zeta_2 x} dx \right) \otimes \delta_\sigma & \text{when } c < v \le 2c, \\
	\sum_{\sigma  \in \Sigma} \left( d_\sigma  \delta_0 + c_2 a_\sigma \left(\zeta_2\right) e^{\zeta_2 x} dx + c_3 a_\sigma \left(\zeta_3\right) e^{\zeta_3 x} dx \right) \otimes \delta_\sigma & \text{when } v > 2c,
	\end{array}\right.
	$$
	where $c_i, \zeta_i, d_\sigma$ and $a_\sigma(\zeta_i)$ are explicit functions of $\alpha, \beta, c$ and $v$ (see Proposition~\ref{prop:finite_linear_invariant_measure}). In addition
	\begin{itemize}
		\item[\( \bullet \)] if \( c < v \le 2 c\) then \( d_{+2} = 0 \) and \( d_\sigma > 0 \) for \( \sigma \ne +2 \),
		\item[\( \bullet \)] if \( v > 2c \) then \( d_{+2} = d_{+1} = 0 \) and \( d_\sigma > 0 \) for \( \sigma \notin \{ +2, +1\} \).
	\end{itemize}
	\item[(iii)] The unique invariant measure of the instantaneous harmonic process is
	\[
	\pi = \sum_{\sigma \in \Sigma} \left( d_\sigma \delta_0 + \pi_\sigma(x) dx\right) \otimes \delta_\sigma
	\]
	where \( \pi_\sigma \) are defined explicitly in terms of hypergeometric functions (see Section~\ref{sec:instantaneous_harmonic_invariant_measure}). Its support is the compact set \( [0, v/\mu] \times \Sigma \). Finally
	\[
	d_2 = 0 \text{ while } d_0, d_{-2} > 0,
	\]
	and the behavior of the density part of the \(x\)-marginal \(\sum_\sigma \pi_\sigma(x)\) is as follows  
	\begin{itemize}
		\item[\( \bullet \)] at \(x = 0\), it diverges when \(\mu \geq \omega\) and converges to a positive limit when \(\mu < \omega\),
		\item[\( \bullet \)] at \(x = v/\mu\), it diverges when \(\mu > \omega\), converges to a positive limit when \(\mu = \omega\) and vanishes when \(\mu < \omega\).  
	\end{itemize}
\end{itemize}
\end{Thm}

Observe that, in the bulk $\{x > 0\}$, the instantaneous linear invariant measure is a re-weighting of the invariant measure of the process in~\cite[Section III.A]{ledoussal21}. This is because the processes have the same behavior in the bulk and they enter the bulk the same way (always through the point~\mbox{$(x, \sigma) = (0, 2)$}). However, the coefficients of the Dirac masses on the boundary~$\{x = 0\}$ differ. Furthermore, all the processes in the present article have state space $\mathbb R_+ \times \Sigma$ instead of $\mathbb R \times \Sigma$ since the particles cannot cross and therefore their inter-particle distance remains non-negative. Also note that, unlike in~\cite[Section IV]{ledoussal21}, the invariant measures in the present article are all unique. As in~\cite{hahn23}, the central observation that allows us to determine the invariant measure of the instantaneous linear process is that, in the bulk, Fokker-Planck takes the form of a system of ordinary linear differential equations. This implies that the invariant measure is given by a matrix exponential.

It is noteworthy that the richer, biologically relevant~\cite{saragosti12,slowman17}, relative velocity transition rates~\ref{fig:finite_tumble_relative_particle_velocity_transition_rates} lead to different formulae for the invariant measure depending on whether $c \le v \le 2c$ or $v > 2c$. This behavior comes from the fact that when $(x, \sigma) = (0, 1)$
\begin{itemize}
\item if $c < v \le 2c$ then the attractive interaction  between the particles is stronger than the self-propulsion so the process is `glued' to $(0, 1)$ until a velocity change occurs, 
\item if $v > 2c$ then the self-propulsion outweighs the attractive interaction and the particles immediately separate.
\end{itemize}
Hence the nature of the boundary point $(x, \sigma) = (0, 1)$ depends on $c$ and $v$, leading to the presence of a Dirac mass only when $v \le 2c$. Note that, once restricted to the bulk, the invariant measure is a product measure for \( v \le 2c \) but not for \( v > 2c \). These important qualitative differences depending on whether \( \frac v c \le 2 \) or \( \frac v c > 2\) are reminiscent of the universality classes in~\cite{hahn23} and the shape transition in~\cite{dhar19,basu20}. The strategy to determine the invariant measure remains the same as for the instantaneous linear process.

For the same reasons as the instantaneous linear process, the invariant measure of the instantaneous harmonic process is a re-weighting of the invariant measure in~\cite{basu20} in the bulk \( \{ x > 0 \} \) but has a different state space and Dirac masses on the boundary \( \{ x = 0 \} \). Hence it has compact support and displays the same shape transition as in~\cite{basu20}. \\

We also quantify the speed of convergence towards the invariant measure. In the linear potential case, we delay the statement of precise non-asymptotic convergence bounds until Theorem~\ref{thm:long_time_behavior_linear_process} and instead give asymptotic bounds which are easier to state and interpret.

\begin{Thm}\label{thm:convergence_of_all_processes}
\begin{itemize}
\item[(i)] In the instantaneous linear case
$$
\frac{1}{2} \omega \frac{c^2}{v^2} \le \varliminf_{t \rightarrow +\infty} -\frac{1}{t} \log \tv{\delta_{(x, \sigma)} P_t  - \pi } \le \varlimsup_{t \rightarrow +\infty} -\frac{1}{t} \log \tv{\delta_{(x, \sigma)} P_t  - \pi } \le 4 \omega \frac{c^2}{v^2}
$$
for all $(x, \sigma) \in E$ where
\[
	\tv{\eta - \tilde\eta} = \inf \left\{ \mathbb P((x, \sigma) \ne (\tilde x, \tilde \sigma)) : (x, \sigma) \sim \eta \text{ and } (\tilde x, \tilde \sigma) \sim \tilde \eta  \right\} 	
\]
is the total variation distance (see Theorem~\ref{thm:long_time_behavior_linear_process}~and Proposition~\ref{prop:explicit_decay_rates}).
\item[(ii)] In the finite linear case
\begin{align*}
\frac{3 - \sqrt{5}}{8} \min\left(\alpha \left(1 + \frac{\alpha}{\beta}\right) \frac{c^2}{v^2}, \alpha \frac{c}{v}\right) &\le \varliminf_{t \rightarrow +\infty} -\frac{1}{t} \log \tv{\delta_{(x, \sigma)} P_t  - \pi } & \\
&\le \varlimsup_{t \rightarrow +\infty} -\frac{1}{t} \log \tv{\delta_{(x, \sigma)} P_t  - \pi } \le 4 \alpha \left( 1 + \frac{\alpha}{\beta} \right) \frac{c^2}{v^2}
\end{align*}
for all $(x, \sigma) \in E$ (see Theorem~\ref{thm:long_time_behavior_linear_process}~and Proposition~\ref{prop:explicit_decay_rates}).
\item[(iii)] In the instantaneous harmonic case, for \( p \ge 1 \)
$$
\min \left( \mu, \frac\omega p\right) \le \varliminf_{t \rightarrow +\infty} -\frac{1}{t} \log \mathcal W_p \left(\delta_{(x, \sigma)} P_t, \pi\right)
$$
for all \( (x, \sigma) \in E\) (see Theorem~\ref{thm:long_time_behavior_harmonic_process_wasserstein} and Remark~\ref{rem:instantaneous_harmonic_asymptoic_decay_rate}) where
$$
\mathcal W_p\left(\eta, \tilde \eta\right) := \inf \left\{ \left(\mathbb E \left|x - \tilde x\right|^p\right)^{\frac{1}{p}} + \mathbb P\left( \sigma \neq \tilde \sigma \right) : (x, \sigma) \sim \eta \text{ and } (\tilde x, \tilde \sigma) \sim \tilde \eta\right\}
$$
is a Wasserstein-type distance (see Definition~\ref{def:mixed_distance}).
\end{itemize}
\end{Thm}

Note that the upper and lower bounds for the exponential decay rate only differ by a constant factor in the instantaneous linear case. Hence we have fully determined its dependence on all model parameters. The lower bound is shown using the same synchronous coupling as in~\cite{guillin24}. However, unlike in this previous work, the non-compactness of the state space makes the mixing time infinite, which is why convergence speed under Dirac initial distributions is considered instead. The upper bound is deduced from concentration inequalities reflecting the fact that the process is slow to explore parts of the state space where the inter-particle distance $x$ is large. Theorem~\ref{thm:long_time_behavior_linear_process} reveals that there is prefactor to the exponential decay, which grows exponentially with~\( x \). It is worth noting that related models have been studied through Foster-Lyapunov techniques~\cite{fontbona16} and hypocoercivity methods~\cite{calvez14}.

In the finite linear case there is a mismatch between the upper and lower bound when
$$
\alpha \left(1 + \frac{\alpha}{\beta}\right) \frac{c^2}{v^2} \gg \alpha \frac{c}{v} \iff \left(1 + \frac{\alpha}{\beta}\right) \frac{c}{v} \gg 1.
$$
Because $0 < c/v < 1$ this can only happen when ${\alpha}/{\beta} \gg 1$ meaning that the particles spend the overwhelming majority of their time tumbling and not running. Hence the mismatch only occurs in the least relevant regime from a modeling perspective. When $\left(1 + \alpha/\beta\right) ({c}/{v})$ is upper bounded, the upper and lower bound differ at most by a constant factor. Note that, in the limit \( \beta \to \infty \), the finite linear process becomes an instantaneous linear process with \( \omega = \alpha/2 \), which is coherent with the rates of Theorem~\ref{thm:convergence_of_all_processes}. The core ideas of the proofs remain the same as for the instantaneous linear process. The prefactor of the decay is again exponential in \( x \).

The exponential ergodicity of the instantaneous harmonic process is shown following the coupling approach in~\cite{benaim12}. The idea is to use the exponential contractivity of the deterministic dynamics to obtain exponential decay in a Wasserstein-type distance. It is noteworthy that, for fixed \( p \ge1 \), this decay rate is, up to a constant factor, the minimum of the deterministic contraction rate \( 2\mu \) and the rate of convergence \( 2 \omega \) of the relative velocity \( \sigma(t) \) towards its invariant measure. Using Foster-Lyapunov techniques~\cite{meyn93,hairer11} would yield exponential decay in total variation distance, but the resulting bounds would be quantitatively poor.

\section{Invariant measures}\label{sec:invariant_measures}

We delay the proof of the existence and uniqueness of invariant probabilities for our models (see Theorem~\ref{thm:long_time_behavior_linear_process} and Theorem~\ref{thm:long_time_behavior_harmonic_process_wasserstein}) and turn to the derivation of exact formulae for these probabilities under the assumption that they exist and are unique.

\subsection{Instantaneous linear process}\label{sec:instantaneous_linear_process}

We start by outlining the general approach we will use to compute the invariant measure of all our models. This method is similar in spirit to~\cite{frydel21,ledoussal21,hahn23} and always follow the same basic steps
\begin{enumerate}
\item write down differential equations for the density of the invariant measure in the bulk~$\mathbb R_{>0} \times \Sigma$,
\item find the general solution of these equations (which contains constants that remain to be fixed),
\item use integrability and symmetry constraints to reduce the number of constants,
\item use the fact that the $\sigma$-marginal of $\pi$ is the invariant measure of the $\Sigma$-valued process with generator $\mathcal Q$ to determine the remaining constants and the weights of the Dirac masses on the boundary.
\end{enumerate}
While these steps can be applied to a wide range of one-dimensional systems, they do not always make it possible to fix all the constants. However, the following additional step always enables the full characterization of the invariant measure
\begin{itemize}
\item[5.] reinject the result of the previous steps into the generator characterization of $\int \mathcal L f d\pi = 0$.
\end{itemize}

Interestingly this approach can also be applied to systems with Gaussian noise~\cite{frydel21, malakar18}. The main challenge is always finding an analytical solution for the system of differential equations in the bulk. The finite harmonic process (i.e.~\( V(x) = \frac\mu2x^2 \) together with the rates of Figure~\ref{fig:finite_tumble_relative_particle_velocity_transition_rates}) provides an example where this seems challenging. With this in mind, we turn to the computation of the invariant measure of the instantaneous linear process.

\begin{Prop}[Invariant measure of the instantaneous linear process] \label{prop:instantaneous_linear_invariant_measure}

The unique invariant measure of the instantaneous linear process is given by
$$
\pi = \sum_{\sigma  \in \Sigma } \left( d_{\sigma } \delta_0 + a_\sigma e^{\zeta x} dx \right) \otimes \delta_{\sigma}
$$
where $\zeta = -\frac{2 c \omega}{v^{2} - c^{2}}$ and
\begin{align*}
d_2 &= 0, & d_0 &= \frac{c}{c+ v}, & d_{-2} &= \frac{cv}{(c + v)^2}, \\
a_2 &= \frac{c \omega}{2  {\left(v^2 - c^2\right)}}, & a_0 &= \frac{c \omega}{{\left(c + v\right)}^{2}}, & a_{-2} &= \frac{ c  {\left(v - c\right)} \omega}{2  {\left(c + v\right)}^{3}}.
\end{align*}
\end{Prop}

\begin{proof} Let $\pi$ be the unique invariant measure of the instantaneous linear process and let $\pi_2, \pi_0, \pi_{-2}$ be the Borel measures on $\mathbb R_{>0}$ such that
$$
\pi\left(1_{\{x>0\}} f \right) = \sum_{\sigma \in \Sigma} \int_0^{+\infty} f(x, \sigma) d\pi_\sigma(x)
$$
for all bounded measurable $f : E \rightarrow \mathbb R$. \\

Let $f_{-2}, f_0, f_2 \in C_c^{\infty}(\mathbb R_{>0})$ be arbitrary but fixed and let the function $f : E \rightarrow \mathbb R$ be defined by $f_\sigma(x, \sigma) = f_\sigma(x)$ if $x > 0$ and $f(x, \sigma) = 0$ when $x = 0$. We have $f \in D(\mathcal L)$ and using the explicit expression of $\mathcal L$ in Proposition~\ref{prop:feller+strong_markov+generator} yields
\begin{align*}
\int \mathcal L f d\pi &= \sum_{\sigma \in \Sigma} \left(\int_0^{+\infty} \left(\left(v \sigma -2c\right) \partial_x f_\sigma(x) + \sum_{\tilde \sigma \in \Sigma} q_{\sigma, \tilde \sigma} f_{\tilde \sigma}(x) \right) d\pi_\sigma(x)\right) \\
&= \sum_{\sigma \in \Sigma} \left[-(v \sigma -2c)\pi_\sigma' + \sum_{\tilde \sigma \in \Sigma} q_{\tilde \sigma, \sigma} \pi_\sigma\right] \left(f_\sigma\right)
\end{align*}
where $\pi_\sigma'$ is the distributional derivative of $\pi_\sigma$. By~\cite[Theorem 3.37]{liggett10} we have $\int \mathcal L f d\pi = 0$ so, because the $f_\sigma$ are arbitrary, we get
$$
- \mathcal V \Pi' + \mathcal Q^t \Pi = 0
$$
where $\mathcal V = v \text{Diag}\left(2, 0, -2\right) -2c I_3$, $\mathcal Q = (q_{\sigma, \sigma'})_{\sigma, \sigma' \in \Sigma}$ is the discrete generator of the Markov jump process followed by the $\sigma$-marginal and $\Pi = \left(\pi_2, \pi_0, \pi_{-2}\right)^t$. Because $\mathcal V$ is invertible, this is equivalent to $\Pi' = \mathcal A \Pi$ where $\mathcal A = \mathcal V^{-1} \mathcal Q^t$. The diagonalization
$$
\mathcal A
=
\left(\begin{array}{rrr}
1 & 1 & 1 \\
-\frac{2   v}{c} & \frac{2   {\left(c + v\right)}}{v - c} & 2 \\
-1 & \frac{(c+v)^2}{(v-c)^2} & 1
\end{array}\right)
\left(\begin{array}{rrr}
\frac{\omega}{c} & 0 & 0 \\
0 & -\frac{2 c \omega}{v^2 - c^{2}} & 0 \\
0 & 0 & 0
\end{array}\right)
\left(\begin{array}{rrr}
1 & 1 & 1 \\
-\frac{2   v}{c} & \frac{2   {\left(c + v\right)}}{v - c} & 2 \\
-1 & \frac{(c+v)^2}{(v-c)^2} & 1
\end{array}\right)^{-1}
$$
implies that the general (distributional) solution of $\Pi' = \mathcal A \Pi$ is
$$
\begin{pmatrix}
\pi_{-2} \\
\pi_{0} \\
\pi_{2}\\
\end{pmatrix}
=
a
\begin{pmatrix}
1 \\
\frac{2 {\left(c + v\right)}}{v - c} \\
\frac{(c+v)^2}{(v-c)^2}
\end{pmatrix}
e^{-\frac{2 c \omega}{v^{2} - c^{2}} x} dx
+ b
\begin{pmatrix}
1 \\
-\frac{2v}{c}\\
-1
\end{pmatrix}
e^{\frac{\omega}{c} x} dx
+c
\begin{pmatrix}
1 \\
2 \\
1
\end{pmatrix}
dx
$$
where $ a, b, c \in \mathbb{R}$ are constants that remain to be fixed. Because \( \pi_\sigma(\mathbb R_{>0}) < +\infty \), we know that $b = c = 0$ so that only $a$ remains to be fixed. \\

Because the $\sigma$-marginal of the process is a Markov jump process with generator $\mathcal Q$ we have
$$
\pi \left([0, +\infty) \times \{2\}\right) = \frac{1}{4}, \quad \pi \left([0, +\infty) \times \{0\}\right) = \frac{1}{2}, \quad \pi \left([0, +\infty) \times \{-2\}\right) = \frac{1}{4}.
$$

The measure $\pi$ is unique so it is ergodic and almost surely
$$
\lim_{t \rightarrow +\infty} \frac{1}{t} \int_0^t 1_{\{X(s) = (0, 2)\}}ds = \pi \left(\{ (0, 2) \}\right).
$$
Because the process instantaneously leaves the point $(0, 2)$ and cannot come back before a stochastic velocity change occurs, we have $\int_0^t {1}_{\{X(s) = (0, 2)\}} ds = 0$ for all $t \ge 0$ almost surely so that $\pi(\{(0, 2)\}) = 0$. \\ 
 
Setting $\zeta := -\frac{2c\omega}{v^2 - c^2}$ we get
$$
\frac{1}{4} = \pi \left([0, +\infty) \times \{2\}\right) = a \int_0^{+\infty} e^{\zeta x} dx = \frac{a}{-\zeta} \implies a = \frac{c \omega}{2 (v^2 - c^2)},
$$
and thus
$$
\pi_2 = \frac{c \omega}{2 (v^2 - c^2)} e^{\zeta x} dx, \quad \pi_0 = \frac{c\omega}{(c+v)^2}e^{\zeta x}dx, \quad \pi_{-2} = \frac{c(v-c)\omega}{2(c + v)^3} e^{\zeta x}dx.
$$

Finally
\begin{align*}
\pi \left( \{(0, 0)\}\right) &= \frac{1}{2} - \pi_0 \left( \mathbb R_{>0} \right) = \frac{c}{c+v}, \\
\pi \left( \{ (0, -2) \} \right) &= \frac{1}{4} - \pi_{-2} \left( \mathbb R_{>0} \right) = \frac{cv}{(c+v)^2}
\end{align*}
which concludes our explicit characterization of $\pi$.
\end{proof}

\subsection{Finite linear process}\label{subsec:finite_linear_invariant_measure}

Recall that when \( \sigma(t) \) follows~\ref{fig:finite_tumble_relative_particle_velocity_transition_rates} its transition-rate matrix is given by~\eqref{eq:finite_Q} and define
\[
	\mathcal V = \bordermatrix{ & \tilde \sigma = 2 & \tilde \sigma = 1 & \tilde \sigma = 0_\pm & \tilde \sigma = 0_0 & \tilde \sigma = -1 & \tilde \sigma = -2 \cr
\sigma = 2     & 2v  -2c & 0 & 0 & 0 & 0 & 0 \cr
\sigma = 1     & 0 & v  -2c & 0 & 0 & 0 & 0 \cr
\sigma = 0_\pm & 0 & 0 & -2c & 0 & 0 & 0 \cr
\sigma = 0_0   & 0 & 0 & 0 & -2c & 0 & 0 \cr
\sigma = -1    & 0 & 0 & 0 & 0 & -v  -2c & 0 \cr
\sigma = -2    & 0 & 0 & 0 & 0 & 0 & -2v  -2c
}.
\]

To find the invariant measure of the finite linear process, we will apply the same strategy as in the instantaneous case. As before, the key will be solving the linear ODE $-\mathcal V \Pi' + \mathcal Q^t \Pi = 0$ with the added difficulty of increased matrix size and non-invertibility of $\mathcal V$ when $v = 2c$.

The following Lemma~\ref{lem:roots_of_P2_and_P3} and Lemma~\ref{lem:eigenvector} are needed to write down the invariant measure of the finite linear process. They provide, in essence, the spectral decomposition of \mbox{\( \mathcal A = \mathcal V^{-1} \mathcal Q \)}, which plays the same central role as in the proof of~\ref{prop:instantaneous_linear_invariant_measure}. Their proofs are postponed until the end of the section.

We have
$$
\det(\mathcal V) \det(-X I + \mathcal A) = \det(-X \mathcal V + 	\mathcal Q^t) = -8 X P_2 P_3
$$
where the polynomials $P_2$ and $P_3$ are given by
\begin{align*}
P_2 &= c(v^2-c^2) X^2 + ((2\alpha+\beta)c^2 - \beta v^2) X - \alpha(\alpha+\beta)c, \\
P_3 &= c\left(8 c^2 - 2 v^{2} \right) X^3 + 2\left( {\alpha} v^{2} - {\left(8   {\alpha} + 4{\beta}\right)} c^{2} \right) X^2 + 2c \left(5   {\alpha}^{2} + 5   {\alpha} {\beta} + {\beta}^{2} \right) X - \left( 2   {\alpha}^{3} + 3   {\alpha}^{2} {\beta} + {\alpha} {\beta}^{2} \right).
\end{align*}
Their roots will play a crucial role in the computation of the invariant measure. 

\begin{Lem}\label{lem:roots_of_P2_and_P3} For \( v > c > 0 \) and \( \alpha, \beta > 0 \) it holds that
	\begin{itemize}
		\item[\normalfont(i)] $P_2$ is a quadratic polynomial with one positive and one negative root,
		\item[\normalfont(ii)] if $c < v < 2c$ then $P_3$ is a cubic polynomial with no negative roots,  
		\item[\normalfont(iii)] if $v = 2c$ then $P_3$ is a quadratic polynomial with no negative roots,
		\item[\normalfont(iv)] if $v > 2c$ then $P_3$ is a cubic polynomial with a single negative root,
		\item[\normalfont(v)] $P_3$ has only real roots,
		\item[\normalfont(vi)] $P_2$ and $P_3$ have distinct negative roots when $v > 2c$. 
	\end{itemize}
\end{Lem}

\begin{Not}\label{not:zeta2_zeta3_piQ} The following notations will be useful
\begin{itemize}
\item denote $\zeta_2$ the unique negative root of $P_2$,
\item denote $\zeta_3$ the unique negative root of $P_3$ when $v > 2c$,
\item denote $\pi_{\mathcal Q}$ the invariant measure of the Markov jump process with transition-rate matrix $\mathcal Q$, which is given by
$$
\begin{pmatrix}
\pi_{\mathcal Q} \left(\{ 2 \}\right) \\
\pi_{\mathcal Q} \left(\{ 1 \}\right) \\
\pi_{\mathcal Q} \left(\{ 0_\pm \}\right) \\
\pi_{\mathcal Q} \left(\{ 0_0 \}\right) \\
\pi_{\mathcal Q} \left(\{ -1 \}\right) \\
\pi_{\mathcal Q} \left(\{ -2 \}\right)
\end{pmatrix} =
\frac{1}{(\alpha+\beta)^2}
\begin{pmatrix}
\frac{1}{4}{\beta}^2 \\
{\alpha}{\beta} \\
\frac{1}{2}{\beta}^2 \\
{\alpha}^2 \\
{\alpha}{\beta} \\
\frac{1}{4}{\beta}^2
\end{pmatrix}.
$$
\end{itemize}
\end{Not}

The next lemma gives explicit formulae for the eigenvectors corresponding to $\zeta_2$ and $\zeta_3$.

\begin{Lem} \label{lem:eigenvector} If $\zeta = \zeta_2$ (resp. $v > 2c$ and $\zeta = \zeta_3$) the kernel of $-\zeta \mathcal V + \mathcal Q^t$ is spanned by \( \mathrm a (\zeta) \) defined as
$$
	\begin{pmatrix}
	a_2(\zeta) \\
	a_1(\zeta) \\
	a_{0_\pm}(\zeta) \\
	a_{0_0}(\zeta) \\
	a_{-1}(\zeta) \\
	a_{-2}(\zeta)
	\end{pmatrix}
	=
	\begin{pmatrix}
	-{\left(3  c {\zeta} - 2  {\alpha} - {\beta}\right)} {\alpha}^{2} {\beta}^{2} \\
4  {\left(3  c {\zeta} - 2  {\alpha} - {\beta}\right)} {\left(c {\zeta} - v {\zeta} - {\alpha}\right)} {\alpha}^{2} {\beta} \\
-2  {\left(4  c^{2} {\zeta}^{2} - 6  c v {\zeta}^{2} + 2  v^{2} {\zeta}^{2} - 6  {\alpha} c {\zeta} - 2  {\beta} c {\zeta} + 4  {\alpha} v {\zeta} + 2  {\beta} v {\zeta} + 2  {\alpha}^{2} + {\alpha} {\beta}\right)} {\left(c {\zeta} - {\beta}\right)} {\alpha} {\beta} \\
-4  {\left(4  c^{2} {\zeta}^{2} - 6  c v {\zeta}^{2} + 2  v^{2} {\zeta}^{2} - 6  {\alpha} c {\zeta} - 2  {\beta} c {\zeta} + 4  {\alpha} v {\zeta} + 2  {\beta} v {\zeta} + 2  {\alpha}^{2} + {\alpha} {\beta}\right)} {\left(c {\zeta} - {\alpha}\right)} {\alpha}^{2} \\
	4 \rho(\zeta) {\left(3  c {\zeta} - 2  {\alpha} - {\beta}\right)} {\left(c {\zeta} + v {\zeta} - {\alpha}\right)} {\alpha}^{2} {\beta} \\
-\rho(\zeta){\left(3  c {\zeta} - 2  {\alpha} - {\beta}\right)} {\alpha}^{2} {\beta}^{2}
	\end{pmatrix}
	$$
	where
	$$
	\rho(\zeta) = \frac{4  c^{2} {\zeta}^{2} - 6  c v {\zeta}^{2} + 2  v^{2} {\zeta}^{2} - 6  {\alpha} c {\zeta} - 2  {\beta} c {\zeta} + 4  {\alpha} v {\zeta} + 2  {\beta} v {\zeta} + 2  {\alpha}^{2} + {\alpha} {\beta}}{4  c^{2} {\zeta}^{2} + 6  c v {\zeta}^{2} + 2  v^{2} {\zeta}^{2} - 6  {\alpha} c {\zeta} - 2  {\beta} c {\zeta} - 4  {\alpha} v {\zeta} - 2  {\beta} v {\zeta} + 2  {\alpha}^{2} + {\alpha} {\beta}}.
	$$
\end{Lem}

We can now give explicit formulae for the invariant measure of the finite linear process.

\begin{Prop}[Invariant measure of the finite linear process]\label{prop:finite_linear_invariant_measure} \mbox{}
\begin{itemize}
\item[\normalfont(i)] {When $c < v \le 2c$}, the unique invariant measure of the finite linear process is given by
$$
\pi = \sum_{\sigma  \in \Sigma} \left( d_\sigma  \delta_0 + c_2 a_{\sigma }\left(\zeta_2\right) e^{\zeta_2 x} dx \right) \otimes \delta_\sigma
$$
where
$$
c_2 = \frac{-{\zeta_2}}{4  {\left(-3  c {\zeta_2} + 2  {\alpha} + {\beta}\right)} {\left({\alpha} + {\beta}\right)}^{2} {\alpha}^{2}}
$$
and
$$
d_\sigma = \left\{ \begin{array}{cl}
0 &\text{if } \sigma = 2, \\
{\pi_{\mathcal Q}}(\{\sigma\}) + \frac{c_2 a_\sigma(\zeta_2)}{\zeta_2} &\text{if } \sigma \ne 2.
\end{array}\right.
$$
\item[\normalfont(ii)] {When $v > 2c$}, the unique invariant measure of the finite linear process is given by
$$
\pi = \sum_{\sigma  \in \Sigma} \left( d_\sigma  \delta_0 + c_2 a_\sigma \left(\zeta_2\right) e^{\zeta_2 x} dx + c_3 a_\sigma \left(\zeta_3\right) e^{\zeta_3 x} dx \right) \otimes \delta_\sigma
$$
where
\begin{align*}
c_2 &= -\frac{{\zeta_2} {\zeta_3}}{4 \alpha^2 (\alpha+\beta)^2{\left(3  c {\zeta_2} - 2  {\alpha} - {\beta}\right)} {\left({\zeta_2} - {\zeta_3}\right)}}, \\
c_3 &= \frac{{\zeta_2} {\zeta_3}}{4 \alpha^2(\alpha+\beta)^2 {\left(3  c {\zeta_3} - 2  {\alpha} - {\beta}\right)} {\left({\zeta_2} - {\zeta_3}\right)}},
\end{align*}
and
$$
d_\sigma = \left\{
\begin{array}{cl}
0 &\text{if } \sigma = 1, 2, \\
{\pi_{\mathcal Q}}(\{\sigma\}) + \frac{c_2 a_\sigma(\zeta_2)}{\zeta_2} +\frac{c_3 a_\sigma(\zeta_3)}{\zeta_3} &\text{if } \sigma \ne 1, 2.
\end{array}
\right.
$$
\end{itemize}
\end{Prop}

Note that the invariant measure displays a rich behavior and differs when $c < v \le 2c$ and when $v > 2c$. Because $\pi\left(\mathbb R_+ \times \{\sigma\}\right) < +\infty$ only the negative eigenvalues of $\mathcal A = \mathcal V^{-1} \mathcal Q^t$ can enter the invariant measure. Hence, the different formulae can be explained by the fact that the characteristic polynomial of $\mathcal A$ equals $-\frac{8}{\det (\mathcal V)}XP_2P_3$ and
\begin{itemize}
\item when $c < v \le 2c$ the polynomial $P_2$ has a unique negative root while $P_3$ has no negative root,
\item when $v > 2c$ the polynomials $P_2$ and $P_3$ each have a single negative root,
\end{itemize}
as stated in Lemma~\ref{lem:roots_of_P2_and_P3}. The process also differs qualitatively. As outlined in the introduction 
\begin{itemize}
	\item when $c < v \le 2c$ the process stays `glued' to \( (0, 1) \) (until a velocity change occurs) whenever it visits the state,
	\item when $v > 2c$ the process immediately leaves the state \( (0, 1) \) whenever it passes through it.
\end{itemize}
As a result $\pi\left(\{(0, 1)\}\right) > 0$ when $c < v \le 2c$ and $\pi\left(\{(0, 1)\}\right) = 0$ when $v > 2c$.
\begin{proof}[Proof of Proposition~\ref{prop:finite_linear_invariant_measure}] Let $\pi$ be the unique invariant measure of the finite linear process and for $\sigma \in \Sigma$ let $(\pi_\sigma)_{\sigma \in \Sigma}$ be the Borel measures on $\mathbb R_{>0}$ such that
$$
\pi\left(1_{\{x>0\}} f \right) = \sum_{\sigma \in \Sigma} \int_0^{+\infty} f(x, \sigma) d\pi_\sigma(x)
$$
for all bounded measurable $f : E \rightarrow \mathbb R$. \\

The same argument as in the proof of Proposition~\ref{prop:instantaneous_linear_invariant_measure} yields
\begin{equation}\label{eq:bulk_ode}
-\mathcal V \Pi' + \mathcal Q^t \Pi = 0
\end{equation}
in the distributional sense, where $\Pi = (\pi_\sigma)_{\sigma \in \Sigma}$. \\

\underline{Case $v \ne 2c$.} In this case, $\mathcal V$ is invertible so~\eqref{eq:bulk_ode} is equivalent to $\Pi' = \mathcal A \Pi $ where $\mathcal A = \mathcal V^{-1} \mathcal Q^t$. Hence solving~\eqref{eq:bulk_ode} boils down to the spectral analysis of $\mathcal A$. \\

If we denote
\begin{itemize}
\item $K$ the number of Jordan blocks of $\mathcal A$,
\item $N_k$ the size of the $k$-th Jordan block,
\item $\lambda_k$ the eigenvalue associated to the $k$-th Jordan block,
\item $(\mathrm a_{k, n})_{n = 0}^{N_k-1}$ the \( k \)-th Jordan chain of generalized eigenvectors, i.e.
$$
\mathcal A \mathrm a_{k, 0} = \lambda_k \mathrm a_{k, 0} \text{ and } \mathcal A \mathrm a_{k, n} = \lambda_k \mathrm a_{k, n} + \mathrm a_{k, n-1} \text{ for } n \ge 1,
$$
\end{itemize}
then 
$$
e^{t \mathcal A} \left( \sum_{k = 1}^K \sum_{n = 0}^{N_k - 1} b_{k, n} \mathrm a_{k, n}  \right) = \sum_{k=1}^K e^{\lambda_k t} \sum_{n = 0}^{N_k - 1} P_k^{(n)}(t) \mathrm a_{k, n}
$$
where $P_k = \sum_{n = 0}^{N_k - 1} b_{k, n} \frac{X^n}{n!}$ and \( P^{(n)} \) is its \( n \)-th derivative. \\

We have
\begin{equation}\label{eq:char_poly}
\det(\mathcal A - X I) = -\frac{8}{\det(\mathcal V)} X P_2 P_3
\end{equation}
so Lemma~\ref{lem:roots_of_P2_and_P3} implies that all eigenvalues of $\mathcal A$ are real. Because the $\pi_\sigma$ are integrable we have $b_{k, n} = 0$ for all $k$ such that $\lambda_k \ge 0$ and all $n \in \{0, \ldots, N_k - 1\}$. Thus we need only consider the negative eigenvalues of $\mathcal A$ and the associated generalized eigenvectors. \\

\underline{Subcase $c < v < 2c$.} Lemma \ref{lem:roots_of_P2_and_P3} and~\eqref{eq:char_poly} imply that $\zeta_2$ is the only negative eigenvalue of $\mathcal A$. It has multiplicity one and
$$
\mathcal V \text{ invertible} \implies  \ker \left( \mathcal A - \zeta_2 I \right) = \ker \left(-\zeta_2 \mathcal V + \mathcal Q^t \right)
$$
so Lemma~\ref{lem:eigenvector} gives us the associated eigenvector $\mathrm a(\zeta_2)$. Hence $\Pi = c_2 \mathrm a(\zeta_2) e^{\zeta_2 x} dx$ where $c_2$ is a constant that remains to be determined. \\

Because the $\sigma$-marginal of the process is a Markov jump process with generator $\mathcal Q$ we have $\pi \left(\mathbb R_+ \times \{\sigma\}\right) = {\pi_{\mathcal Q}}(\{\sigma\})$ for $\sigma \in \Sigma$. The ergodicity argument in the proof of Proposition~\ref{prop:instantaneous_linear_invariant_measure} shows $d_2 = \pi(\{(0, 2)\}) = 0$ so
\begin{align*}
\pi(\mathbb R_{>0} \times \{2\}) = {\pi_{\mathcal Q}}(\{2\}) &\iff c_2 \frac{a_2(\zeta_2)}{-\zeta_2} = \frac{\beta^2}{4(\alpha + \beta)^2} \\
&\iff c_2 = \frac{-{\zeta_2}}{4  {\left(-3  c {\zeta_2} + 2  {\alpha} + {\beta}\right)} {\left({\alpha} + {\beta}\right)}^{2} {\alpha}^{2}}.
\end{align*}
For $\sigma \ne 2$ we have
$$
\pi(\{(0, \sigma)\}) + \pi(\mathbb R_{>0} \times \{\sigma\}) = {\pi_{\mathcal Q}}(\{\sigma\}) \iff d_\sigma = \pi(\{(0, \sigma)\}) = {\pi_{\mathcal Q}}(\{\sigma\}) + \frac{c_2 a_\sigma(\zeta_2)}{\zeta_2}.
$$

\underline{Subcase $v > 2c$.} Lemma~\ref{lem:roots_of_P2_and_P3} and~\eqref{eq:char_poly} imply that $\zeta_2$ and $\zeta_3$ are the only negative eigenvalues of $\mathcal A$. Both have multiplicity one and Lemma~\ref{lem:eigenvector} yields the associated eigenvectors. Therefore $\Pi = c_2 \mathrm a(\zeta_2)e^{\zeta_2 x} dx + c_3 \mathrm a (\zeta_3) e^{\zeta_3 x} dx$ where $c_2$ and $c_3$ are constants that remain to be fixed. \\

The usual ergodicity argument implies $d_2 = \pi(\{(0, 2)\}) = 0$ and $d_1 = \pi(\{(0, 1)\}) = 0$. Furthermore $\pi\left(\mathbb R_+ \times \{\sigma\}\right) = {\pi_{\mathcal Q}}(\{\sigma\})$ so
$$
\pi(\mathbb R_{>0} \times \{2\}) = {\pi_{\mathcal Q}}(\{2\}) \text{ and } \pi(\mathbb R_{>0}\times \{1\}) = {\pi_{\mathcal Q}}(\{1\})
$$
which can be rewritten as $M \begin{pmatrix}
c_2 \\ c_3 \end{pmatrix} = \begin{pmatrix}
{\pi_{\mathcal Q}}(\{2\}) \\ {\pi_{\mathcal Q}}(\{1\})
\end{pmatrix}$ with $M = \begin{pmatrix}
\frac{a_2(\zeta_2)}{-\zeta_2} & \frac{a_2(\zeta_3)}{-\zeta_3}\\
\frac{a_1(\zeta_2)}{-\zeta_2} & \frac{a_1(\zeta_3)}{-\zeta_3}
\end{pmatrix}$.

We have
$$
\det(M) = \frac{4  {\left(3  c {\zeta_2} - 2  {\alpha} - {\beta}\right)} {\left(3  c {\zeta_3} - 2  {\alpha} - {\beta}\right)} {\alpha}^{4} {\beta}^{3} {\left(c - v\right)} {\left({\zeta_2} - {\zeta_3}\right)}}{{\zeta_2} {\zeta_3}} \neq 0
$$
so we can solve the system and obtain
\begin{align*}
c_2 &= -\frac{{\zeta_2} {\zeta_3}}{4 \alpha^2 (\alpha+\beta)^2{\left(3  c {\zeta_2} - 2  {\alpha} - {\beta}\right)} {\left({\zeta_2} - {\zeta_3}\right)}}, \\
c_3 &= \frac{{\zeta_2} {\zeta_3}}{4 \alpha^2(\alpha+\beta)^2 {\left(3  c {\zeta_3} - 2  {\alpha} - {\beta}\right)} {\left({\zeta_2} - {\zeta_3}\right)}}.
\end{align*}

For $\sigma \ne 1, 2$ we have
$$
\pi(\{(0, \sigma)\}) + \pi(\mathbb R_{>0} \times \{\sigma\}) = {\pi_{\mathcal Q}}(\{\sigma\}) \iff d_\sigma = \pi\left(\{(0, \sigma)\}\right)= {\pi_{\mathcal Q}}(\{\sigma\}) + \frac{c_2 a_\sigma(\zeta_2)}{\zeta_2} + \frac{c_3 a_\sigma(\zeta_3)}{\zeta_3}.
$$

\underline{Case $v = 2c$.} In this case the matrix $\mathcal V$ is no longer invertible. It follows from Lemma~\ref{lem:roots_of_P2_and_P3} and Lemma~\ref{lem:eigenvector} that $\Pi = c_2 \mathrm a(\zeta_2) e^{\zeta_2 x} dx$ is a solution of~\eqref{eq:bulk_ode}. We now show that the solutions of~\eqref{eq:bulk_ode} that are integrable on $\mathbb R_{>0}$ form a one-dimensional vector space. \\

When $v = 2c$ equation~\eqref{eq:bulk_ode} reads
\begin{align*}
\frac{1}{2}  {\beta} {\pi_{1}(x)} - 2  {\alpha} {\pi_{2}(x)} - {\pi'_{2}(x)} v &= 0, \\
{\beta} {\pi_{0_0}(x)} + {\alpha} {\pi_{0_\pm}(x)} - {\left({\alpha} + {\beta}\right)} {\pi_{1}(x)} + 2  {\alpha} {\pi_{2}(x)} &= 0, \\
-2  {\alpha} {\pi_{0_\pm}(x)} + \frac{1}{2}  {\beta} {\pi_{1}(x)} + \frac{1}{2}  {\beta} {\pi_{-1}(x)} + {\pi'_{0_\pm}(x)} v &= 0, \\
-2  {\beta} {\pi_{0_0}(x)} + {\alpha} {\pi_{1}(x)} + {\alpha} {\pi_{-1}(x)} + {\pi'_{0_0}(x)} v &= 0, \\
{\beta} {\pi_{0_0}(x)} + {\alpha} {\pi_{0_\pm}(x)} - {\left({\alpha} + {\beta}\right)} {\pi_{-1}(x)} + 2  {\alpha} {\pi_{-2}(x)} + 2  {\pi'_{-1}(x)} v &= 0, \\
\frac{1}{2}  {\beta} {\pi_{-1}(x)} - 2  {\alpha} {\pi_{-2}(x)} + 3  {\pi'_{-2}(x)} v &= 0.
\end{align*}
Hence the second line can be used to remove $\pi_1$ from the system of equations and obtain a system of linear differential equations which can be solved. This is equivalent to the observation that~\eqref{eq:bulk_ode} implies $\Pi = M_{5 \rightarrow 6} M_{6 \rightarrow 5} \Pi$ and $(M_{6 \rightarrow 5} \Pi)' = \tilde{\mathcal A} (M_{6 \rightarrow 5} \Pi)$ where $\tilde {\mathcal A} = \left(M_{6 \rightarrow 5} \mathcal V M_{5 \rightarrow 6}\right)^{-1} (M_{6 \rightarrow 5} \mathcal Q^t M_{5 \rightarrow 6})$ and
$$
M_{5 \rightarrow 6} = \begin{pmatrix}
1 & 0 & 0 & 0 & 0 \\
\frac{2  {\alpha}}{{\alpha} + {\beta}} & \frac{{\alpha}}{{\alpha} + {\beta}} & \frac{{\beta}}{{\alpha} + {\beta}} & 0 & 0 \\
0 & 1 & 0 & 0 & 0 \\
0 & 0 & 1 & 0 & 0 \\
0 & 0 & 0 & 1 & 0 \\
0 & 0 & 0 & 0 & 1
\end{pmatrix}, \quad
M_{6 \rightarrow 5} = \begin{pmatrix}
1 & 0 & 0 & 0 & 0 & 0 \\
0 & 0 & 1 & 0 & 0 & 0 \\
0 & 0 & 0 & 1 & 0 & 0 \\
0 & 0 & 0 & 0 & 1 & 0 \\
0 & 0 & 0 & 0 & 0 & 1
\end{pmatrix}.
$$

Assume by contradiction that $\Pi_1$ and $\Pi_2$ are two linearly independent solutions of~\eqref{eq:bulk_ode} that are integrable on $\mathbb R_{>0}$. We have
$$
\sum_{i = 1, 2} a_i M_{6 \rightarrow 5} \Pi_i = 0 \implies \sum_{i = 1, 2} a_i M_{5 \rightarrow 6	} M_{6 \rightarrow 5} \Pi_i = 0 \implies \sum_{i=1, 2} a_i \Pi_i = 0 \implies a_1 = a_2 = 0
$$
so $\tilde \Pi_1 = M_{6 \rightarrow 5} \Pi_1$ and $\tilde \Pi_2 = M_{6 \rightarrow 5} \Pi_2$ are two linearly independent solutions of $\tilde \Pi' = \tilde {\mathcal A} \tilde \Pi$ that are integrable on $\mathbb R_{>0}$. \\

On the other hand the characteristic polynomial of $\tilde {\mathcal A}$ is given by $\frac{4  }{3  {\left({\alpha} + {\beta}\right)} v^{5}} X {P_2} {P_3}$ so Lemma~\ref{lem:roots_of_P2_and_P3} implies that $\tilde {\mathcal A}$ has a single negative eigenvalue. This is a contradiction. \\

Hence $\Pi = c_2 \mathrm a(\zeta_2) e^{\zeta_2 x} dx$ and we conclude using the same arguments as in the case \mbox{$c < v < 2c$}.
\end{proof}

We end this section by providing the postponed proofs of Lemma~\ref{lem:roots_of_P2_and_P3} and Lemma~\ref{lem:eigenvector}, which, in essence, describe the spectral decomposition of \( \mathcal A \).

\begin{proof}[Proof of Lemma~\ref{lem:roots_of_P2_and_P3}] (i) The leading coefficient and the constant term of $P_2$ have opposite signs. \\

	(ii) When $v \in (c, 2c)$ we have
	\begin{align*}
	\begin{array}{rl}
	2   {\left(2   c + v\right)} {\left(2   c - v\right)} c &> 0 \\
	-16   {\alpha} c^{2} - 8   {\beta} c^{2} + 2   {\alpha} v^{2} = -2 \left(2\sqrt{2} c + v\right) \left(2\sqrt{2}c - v\right)\alpha - 8 c^2 \beta &< 0  \\
	2   {\left(5   {\alpha}^{2} + 5   {\alpha} {\beta} + {\beta}^{2}\right)} c &>0 \\
	-{\left(2   {\alpha} + {\beta}\right)} {\left({\alpha} + {\beta}\right)} {\alpha} &<0
	\end{array}
	\end{align*}
	so $x < 0$ implies $P_3(x) < 0$. \\
	
	(iii) Follows from the same argument as (ii). \\

	(iv) When $v \in (2c, +\infty)$ we have
	\begin{align*}
	2   {\left(2   c + v\right)} {\left(2   c - v\right)} c < 0, \quad 2   {\left(5   {\alpha}^{2} + 5   {\alpha} {\beta} + {\beta}^{2}\right)} c >0, \quad -{\left(2   {\alpha} + {\beta}\right)} {\left({\alpha} + {\beta}\right)} {\alpha}<0,
	\end{align*}
	so by Descartes'~rule of signs $P_3$ has a unique negative root. \\
	
	(v) When $v \ne 2c$ the polynomial $P_3$ is a cubic with discriminant
	\begin{align*}
	\Delta = 128  {\alpha}^{3} {\beta}^{3} c^{6} + 64  {\alpha}^{2} {\beta}^{4} c^{6} + 64  {\alpha}^{6} {\left(c + v\right)}^{2} {\left(c - v\right)}^{2} v^{2} + 128  {\alpha}^{5} {\beta} c^{4} v^{2} + 288  {\alpha}^{2} {\beta}^{4} c^{4} v^{2} &\\
	+ 256  {\alpha} {\beta}^{5} c^{4} v^{2} + 64  {\beta}^{6} c^{4} v^{2} + 32  {\alpha}^{5} {\beta} c^{2} v^{4} + 404  {\alpha}^{4} {\beta}^{2} c^{2} v^{4} + 280  {\alpha}^{3} {\beta}^{3} c^{2} v^{4} &\\
	+ 52  {\alpha}^{2} {\beta}^{4} c^{2} v^{4} + 96  {\alpha}^{5} {\beta} v^{6} + 32  {\left(2  c^{2} + v^{2}\right)} {\alpha}^{4} {\beta}^{2} {\left(c + v\right)}^{2} {\left(c - v\right)}^{2}& > 0.
	\end{align*}

	When $v = 2c$ the polynomial $P_3$ has degree $2$ and discriminant
	$$
	\Delta = {\left(9  {\alpha}^{4} + 10  {\alpha}^{3} {\beta} + 3  {\alpha}^{2} {\beta}^{2} + 2  {\alpha} {\beta}^{3} + {\beta}^{4}\right)} v^{2} > 0.
	$$

	(vi) Assume by contradiction that the negative roots of $P_2$ and $P_3$ are identical. Then the following resultant must vanish
	\begin{align*}
			&\text{Res}_X \left(P_2, P_3\right) = \\
			&\qquad\left( \beta^2 \left( c^2 - v^2 \right)^2 \left(c^2 + v^2\right) + 4\alpha^2 c^4 v^2 + 4 \alpha \beta c^2 v^4 \right)
	\left( \alpha \left(c^2 - 4 v^2\right) + 2\beta \left( c^2 - v^2 \right) \right)
	{\left({\alpha} + {\beta}\right)} {\left({\alpha} - {\beta}\right)} {\alpha} c.
	\end{align*}
	
	All the factors in the previous expression are positive except for $\alpha - \beta$. Hence $\alpha = \beta$ and
	\begin{align*}
	P_2 &= \left({\left(c^{2} - v^{2}\right)} X - 2  {\alpha} c\right) \left(c X - \alpha\right), \\
	P_3 &= 2 \left({\left(4  c^{2} - v^{2}\right)} X^{2} - 8  X {\alpha} c + 3  {\alpha}^{2}\right) \left(c X - \alpha\right).
	\end{align*}
	
	Hence $P_2$ and $P_3$ share the positive root $\frac{\alpha}{c}$ and since
	$$
	\text{Res}_X \left(\frac{P_2}{cX - \alpha}, \frac{P_3}{cX - \alpha}\right) = 6 \alpha^2 \left(c^2 + v^2\right)^2 > 0
	$$
	there can be no shared negative root. This is a contradiction.
\end{proof}

\begin{proof}[Proof of Lemma~\ref{lem:eigenvector}] Set $\mathcal B = -\zeta \mathcal V + \mathcal Q^t$ and denote
	\begin{itemize}
		\item $\mathrm r_k$ the $k$-th row of $\mathcal B$ for $k = 1, \ldots, 6$,
		\item $\mathcal B_k$ the matrix obtained by replacing the $k$-th row of $\mathcal B$ by the $k$-th basis vector $\mathrm e_k$ for $k = 1, \ldots, 6$.
	\end{itemize}
	
	We now prove
	$$
	\mathcal B_k \text{ invertible } \implies \mathrm u_k := \left(\text{Adj}~\mathcal B_k\right) \mathrm e_k \text{ spans } \ker(\mathcal B) = \left( \text{Vect}(r_l)_l\right)^\perp.
	$$
	
	If $\mathcal B_k$ is invertible then
	$$
	{\mathrm u_k} = \left(\text{Adj}~\mathcal B_k\right) \mathrm e_k \implies \mathcal B_k {\mathrm u_k} = \det\left( \mathcal B_k \right) \mathrm e_k \implies \mathrm u_k \ne 0 \text{ and } \left\langle \mathrm r_l, \mathrm u_k \right\rangle = 0 \text{ for } l \ne k.
	$$
	Because $\mathcal B_k$ is invertible we have that the vectors $(\mathrm r_l)_{l \ne k}$ span a vector space of dimension $5$. Thus \( \ker(\mathcal B) \) is of dimension at most one. It remains to show \( \mathrm u_k \in \left( \text{Vect}(r_l)_l\right)^\perp \) i.e.~$\left\langle \mathrm r_k, {\mathrm u_k} \right\rangle = 0$.  Because \mbox{$\det(-X \mathcal V + \mathcal Q^t) = -8 X P_2 P_3$} Lemma~\ref{lem:roots_of_P2_and_P3} ensures that the family $(\mathrm r_l)_{l}$ spans a vector space of dimension at most $5$. We deduce $\mathrm r_k \in \text{Vect}(\mathrm r_l)_{l \ne k}$ and $\langle \mathrm r_k,\mathrm u_k \rangle = 0$. \\
	
	Working with the resultants
	$$
	\text{Res} \left( \det(\mathcal B_1), P_2 \right), \quad \text{Res} \left( \det(\mathcal B_1), P_3 \right), \quad \text{Res} \left( \det(\mathcal B_6), P_2 \right), \quad \text{Res} \left( \det(\mathcal B_6), P_3 \right),
	$$
	as in the proof of Lemma~\ref{lem:roots_of_P2_and_P3} (v) and using $\zeta = \zeta_2$ (resp. $v \in (2c, +\infty)$ and $\zeta = \zeta_3$) yields $\det(\mathcal B_1) \ne 0$ and $\det(\mathcal B_6) \ne 0$. Hence
	\begin{align*}
	\mathrm u_1 = \begin{pmatrix}
	* \\
	* \\
	2  {\left(4  c^{2} {\zeta}^{2} + 6  c v {\zeta}^{2} + 2  v^{2} {\zeta}^{2} - 6  {\alpha} c {\zeta} - 2  {\beta} c {\zeta} - 4  {\alpha} v {\zeta} - 2  {\beta} v {\zeta} + 2  {\alpha}^{2} + {\alpha} {\beta}\right)} {\left(c {\zeta} - {\beta}\right)} {\alpha} {\beta} \\
	4  {\left(4  c^{2} {\zeta}^{2} + 6  c v {\zeta}^{2} + 2  v^{2} {\zeta}^{2} - 6  {\alpha} c {\zeta} - 2  {\beta} c {\zeta} - 4  {\alpha} v {\zeta} - 2  {\beta} v {\zeta} + 2  {\alpha}^{2} + {\alpha} {\beta}\right)} {\left(c {\zeta} - {\alpha}\right)} {\alpha}^{2} \\
	-4  {\left(3  c {\zeta} - 2  {\alpha} - {\beta}\right)} {\left(c {\zeta} + v {\zeta} - {\alpha}\right)} {\alpha}^{2} {\beta} \\
	{\left(3  c {\zeta} - 2  {\alpha} - {\beta}\right)} {\alpha}^{2} {\beta}^{2}
	\end{pmatrix}, \\
	\mathrm u_6 =
	\begin{pmatrix}
	{\left(3  c {\zeta} - 2  {\alpha} - {\beta}\right)} {\alpha}^{2} {\beta}^{2} \\
	-4  {\left(3  c {\zeta} - 2  {\alpha} - {\beta}\right)} {\left(c {\zeta} - v {\zeta} - {\alpha}\right)} {\alpha}^{2} {\beta} \\
	2  {\left(4  c^{2} {\zeta}^{2} - 6  c v {\zeta}^{2} + 2  v^{2} {\zeta}^{2} - 6  {\alpha} c {\zeta} - 2  {\beta} c {\zeta} + 4  {\alpha} v {\zeta} + 2  {\beta} v {\zeta} + 2  {\alpha}^{2} + {\alpha} {\beta}\right)} {\left(c {\zeta} - {\beta}\right)} {\alpha} {\beta} \\
	4  {\left(4  c^{2} {\zeta}^{2} - 6  c v {\zeta}^{2} + 2  v^{2} {\zeta}^{2} - 6  {\alpha} c {\zeta} - 2  {\beta} c {\zeta} + 4  {\alpha} v {\zeta} + 2  {\beta} v {\zeta} + 2  {\alpha}^{2} + {\alpha} {\beta}\right)} {\left(c {\zeta} - {\alpha}\right)} {\alpha}^{2} \\
	* \\
	*
	\end{pmatrix},
	\end{align*}
	both span $\ker(\mathcal B)$ where we omit the entries marked by \( * \) for conciseness. Hence $\mathrm u_1$ and $\mathrm u_6$ are collinear eigenvectors corresponding to the eigenvalue $\zeta$ and the result follows from
	$$
	4  c^{2} {\zeta}^{2} + 6  c v {\zeta}^{2} + 2  v^{2} {\zeta}^{2} - 6  {\alpha} c {\zeta} - 2  {\beta} c {\zeta} - 4  {\alpha} v {\zeta} - 2  {\beta} v {\zeta} + 2  {\alpha}^{2} + {\alpha} {\beta} \ne 0
	$$
	which can again be shown by computing a resultant.
\end{proof}

\subsection{Instantaneous harmonic process}\label{sec:instantaneous_harmonic_invariant_measure}

We now turn to the invariant probability of the instantaneous harmonic process. The key step of writing down and solving Fokker-Planck in the bulk was already done in~\cite{basu20}. Hence, it suffices to determine the Dirac masses on the boundary induced by the jamming. A sketch of proof is included for the sake of completeness. An important feature of the instantaneous harmonic process is that once it enters the set $\left[0, v/\mu\right] \times \Sigma$, it stays there indefinitely. Hence the invariant measure has compact support. Furthermore, the density part of the invariant displays the same shape transition as in~\cite{basu20}. 

\begin{Not} The following special functions are needed to write down the invariant probability of the instantaneous harmonic process
\begin{itemize}
\item $_pF_q$ the hypergeometric function (see~\cite[Equation 16.2.1]{olver10}),
\item $_p \tilde F_q$ the regularized hypergeometric function given by
$$
{_p} \tilde F _q \left(a_1, \ldots, a_p; b_1, \ldots, b_q; z\right) := \frac{{_p} F _q \left(a_1, \ldots, a_p; b_1, \ldots, b_q; z\right)}{\Gamma(b_1) \cdots \Gamma(b_q)},
$$
	\item $K$ Legendre’s complete elliptic integral of the first kind (see~\cite[Equation 19.2.8]{olver10}),
	\item $G_{m,n}^{p,q}$ the Meijer $G$-function (see~\cite[Equation 16.17.1]{olver10}).
\end{itemize}
\end{Not}

\begin{Prop}[Invariant measure of the instantaneous harmonic process]\label{prop:instantaneous_harmonic_invariant_measure}
Let $\pi$ be unique the invariant measure of the instantaneous harmonic process and set \( b := \frac{\omega}{\mu} \).
\begin{itemize}
	\item[(i)] If $b \ne 1$, then the $x$-marginal of $\pi$ is
	$$
	d_0 \delta_0 + C p(x) dx,
	$$
	where
	\begin{align*}
		d_0 &= {2^{3 b+2} (b-1) \Gamma \left(\frac{b+3}{2}\right)} \times \\
		&\qquad\Bigg[8^b (b-1) \left(\pi ^{3/2} b (b+1) \sec \left(\frac{\pi  b}{2}\right) \, _3\tilde{F}_2\left(\frac{1}{2},\frac{1}{2}-b,1-\frac{b}{2};\frac{3}{2},\frac{1}{2}-\frac{b}{2};1\right)+4 \Gamma \left(\frac{b+3}{2}\right)\right)\\
		&\qquad\qquad-\sqrt{\pi } b (b+1) \Gamma \left(\frac{3}{2}-\frac{b}{2}\right) \Gamma (2 b+1) \, _3\tilde{F}_2\left(\frac{3}{2},1-\frac{b}{2},\frac{b+2}{2};\frac{b+3}{2},\frac{b+4}{2};1\right) \Bigg]^{-1},
	\end{align*}
	and
	\begin{align*}
		p(x) =  \, _2\tilde{F}_1\left(\frac{3}{2}-b,1-\frac{b}{2};\frac{3-b}{2};\frac{x^2 \mu ^2}{v^2}\right)-\frac{\Gamma \left(b-\frac{1}{2}\right) \left(\frac{\mu ^2 x^2}{v^2}\right)^{\frac{b-1}{2}} \, _2\tilde{F}_1\left(\frac{1}{2},1-\frac{b}{2};\frac{b+1}{2};\frac{x^2 \mu ^2}{v^2}\right)}{\sqrt{\pi }},
	\end{align*}
	and
	\begin{align*}
C = \frac{2 \sqrt{\pi } (1 - d_0) \frac{\mu}{v} }{\pi  \, _3\tilde{F}_2\left(\frac{1}{2},\frac{3}{2}-b,1-\frac{b}{2};\frac{3}{2},\frac{3}{2}-\frac{b}{2};1\right) - \Gamma \left(b-\frac{1}{2}\right) \Gamma \left(\frac{b}{2}\right) \, _3\tilde{F}_2\left(\frac{1}{2},1-\frac{b}{2},\frac{b}{2};\frac{b+1}{2},\frac{b+2}{2};1\right)}.
	\end{align*}	
	\item[(ii)] {If $b = 1$}, then the $x$-marginal of the invariant measure is
	$$
	\frac{8}{8+ \pi^2} \delta_0 + \frac{4 \mu}{v (8 + \pi^2)} K \left(1 - \frac{x^2 \mu^2}{v^2}\right) dx.
	$$
\end{itemize}
\end{Prop}

\begin{proof}[Sketch of proof]
	Following~\cite{basu20} we set $b = \frac{\omega}{\mu}$ as well as
	$$
	P = \pi_2 + \pi_0 + \pi_{-2}, \quad Q = \pi_2+ \pi_{-2}, \quad R = \pi_2 - \pi_{-2},
	$$
	where $\pi_2, \pi_0, \pi_{-2}$ are the measures on \( (0, +\infty) \) such that
	$$
	\pi\left(1_{\{x>0\}}f\right) = \sum_{\sigma \in \Sigma} \int_0^{+\infty} f(x, \sigma) d\pi_\sigma(x).
	$$
	
	It is shown in~\cite{basu20} that when $b \ne 1$
	\begin{align*}
		P &= C_1 \Bigg( {_2F_1}\left(\frac{3}{2}-b ,1-\frac{b }{2};\frac{3-b
	}{2};\frac{x^2 \mu ^2}{v^2}\right)+\\
&\quad \quad \frac{2 \Gamma \left(\frac{3-b
	}{2}\right) \Gamma \left(b +\frac{1}{2}\right) \left(\frac{\mu ^2
		x^2}{v^2}\right){}^{\frac{b -3}{2}+1}   {_2F_1}\left(\frac{b
		-3}{2}-b +\frac{5}{2},\frac{b -3}{2}-\frac{b }{2}+2;\frac{b
		-3}{2}+2;\frac{x^2 \mu ^2}{v^2}\right)}{\sqrt{\pi } (1-2 b ) \Gamma
	\left(\frac{b +1}{2}\right)} \Bigg) dx, \\
		Q &= C_1 \Bigg(\frac{(1-b ) {_2F_1}\left(\frac{1}{2}-b ,1-\frac{b
	}{2};\frac{1-b }{2};\frac{x^2 \mu ^2}{v^2}\right)}{1-2 b
}+\\
&\quad \quad \frac{2 \Gamma \left(\frac{3-b}{2}\right) \Gamma \left(b
	+\frac{1}{2}\right) \left(\frac{\mu ^2 x^2}{v^2}\right){}^{\frac{b
			-1}{2}+1}   {_2F_1}\left(\frac{b -1}{2}-b +\frac{3}{2},\frac{b
		-1}{2}-\frac{b }{2}+2;\frac{b -1}{2}+2;\frac{x^2 \mu
		^2}{v^2}\right)}{\sqrt{\pi } (1-2 b ) (b +1) \Gamma
	\left(\frac{b +1}{2}\right)}
\Bigg) dx,
	\end{align*}
	where $C_1$ is a constant that remains to be fixed. It is also shown that $R = \frac{\mu x}{v}P$ so that
	\begin{align*}
	\pi_2 = \frac{1}{2} \left( Q + \frac{\mu x}{v} P \right), \quad \pi_0 = P - Q, \quad \pi_{-2} = \frac{1}{2} \left( Q - \frac{\mu x}{v} P \right).
	\end{align*}
	
	Considering the $\sigma$-marginal yields
	$$
	\pi\left(\mathbb R_+ \times \{2\}\right) = \frac{1}{4}, \quad \pi\left(\mathbb R_+ \times \{0\}\right) = \frac{1}{2}, \quad \pi\left(\mathbb R_+ \times \{-2\}\right) = \frac{1}{4},
	$$
	and $\pi\left(\{(0, 2)\}\right) = 0$ follows from the usual ergodicity argument.
Denoting $C_2 = \pi\left(\{(0, 0)\}\right)$ and $C_3 = \pi \left( \{(0, -2)\} \right)$ this translates to
	\begin{equation}\label{eq:marginal_weights}
	\left.
	\begin{array}{rl}
	\frac{1}{2} \left( Q + \frac{\mu x}{v} P\right) \left(\mathbb R_{>0}\right) &= \frac{1}{4} \\
	C_2 + \left(P - Q\right)(\mathbb R_{>0}) &= \frac{1}{2} \\
	C_3 + \left( Q - \frac{\mu x}{v} P \right) (\mathbb R_{>0}) &= \frac{1}{4}
	\end{array}\right\}
	\end{equation}
	which allows us to determine $C_1, C_2$ and $C_3$. \\
	
	When $b = 1$, it is shown in~\cite{basu20} that 
	$$
	P = \frac{2 C_1}{\pi} K\left(1-\frac{x^2 \mu ^2}{v^2}\right) dx, \quad Q = C_1 G_{2,2}^{2,0}\left(\frac{\mu ^2 x^2}{v^2}\Big|
	\begin{array}{c}
	\frac{1}{2},\frac{3}{2} \\
	0,1 \\
	\end{array}
	\right) dx
	$$
	where $C_1$ is a constant that remains to be fixed. We can again use~\eqref{eq:marginal_weights} to compute the $C_i$ and deduce the desired result.
\end{proof}

\begin{Rem}
The identities
$$
\pi_2 = \frac{1}{2} \left( Q + \frac{\mu x}{v} P \right), \quad \pi_0 = P - Q, \quad \pi_{-2} = \frac{1}{2} \left( Q - \frac{\mu x}{v} P \right)
$$
fully determine $\pi$, not just its $x$-marginal. However, the formulae for the complete invariant probability are somewhat unwieldy and are thus omitted.
\end{Rem}

\section{Convergence towards the invariant measure}\label{sec:convergence}

Having found an explicit representation for the invariant probability of each process naturally leads to the study of the relaxation towards it. This will answer the crucial question: when does the asymptotic behavior take over? Because the processes in this article are non-reversible, the exponential decay rate provided by the spectral approach, which is the most prominent method in the statistical physics literature, is only valid asymptotically. Using coupling instead yields non-asymptotic upper bounds on the distance to the invariant measure, but is not guaranteed to capture the correct speed of convergence. The optimality of our results, in the form of converse bounds of the same order, then has be derived by a different method, which relies on identifying obstacles to mixing.

\subsection{Instantaneous and finite linear process} \label{subsec:convergence_rates_of_instantaneous_and_finite_linear process}

We start by establishing quantitative upper bounds for $\tv{\delta_{(x, \sigma)} P_t - \pi}$ using coupling and large deviation techniques that allow for a unified treatment of the instantaneous and finite linear process. In fact, our computations could be extended to relative velocity transition rates other than \ref{fig:instantaneous_tumble_relative_particle_velocity_transition_rates} and~\ref{fig:finite_tumble_relative_particle_velocity_transition_rates}.

The key is to couple the processes in the same `synchronous' manner as in~\cite{guillin24}, meaning that after some time the velocities of the two copies $X(t) = (x(t), \sigma(t))$ and $\tilde X(t) = (\tilde x(t), \tilde \sigma(t))$ are always identical. Once the velocities are the same, the order between $x$ and $\tilde x$ is preserved, which can be leveraged as follows. If, for example, $x(t) \le \tilde x(t)$ for all $t \ge T$ then $\tilde x(t_0) = 0$ for some $t_0 \ge T$ implies
$$
0 \le x(t_0) \le \tilde x(t_0) = 0 \implies x(t_0) = \tilde x(t_0).
$$
Hence it suffices to prove that $\mathbb P \left(\tau_0 > t\right)$ decays exponentially where $\tau_0 = \inf \{ t > 0 : x(t) = 0\}$. The main difference with the proof strategy in~\cite{guillin24} is the observation that
$$
\tau_0 > t \implies x(t) = x(0) - 2c t + v \int_0^t \sigma(s) ds > 0 \implies \frac{1}{t} \int_0^t \frac{\sigma(s)}{2} ds > \frac{c}{v} - \frac{x(0)}{2vt}
$$
so we can use large deviations results for additive functionals of Markov processes to obtain exponential decay. Furthermore, because the state space is not compact like in~\cite{guillin24}, the mixing time is infinite. Hence we turn to convergence speed under Dirac initial distributions instead.

\begin{Not}\label{not:ldp} Denote
\begin{itemize}
	\item $\mathcal Q$ the transition-rate matrix of figure \ref{fig:instantaneous_tumble_relative_particle_velocity_transition_rates} (resp. figure~\ref{fig:finite_tumble_relative_particle_velocity_transition_rates}),
	\item $\pi_{\mathcal Q}$ the associated invariant probability,
	\item $\tilde {\mathcal V} = \left(\frac12 \sigma {1}_{\{\sigma = \tilde \sigma\}}\right)_{\sigma, \tilde \sigma \in \Sigma }$ the matrix with the normalized velocities on the diagonal,
	\item $\mathcal S = \left( \pi_{\mathcal Q} \left( \sigma\right)1_{\{\sigma = \tilde \sigma \}} \right)_{\sigma, \tilde \sigma \in \Sigma}$ the matrix with the invariant measure on the diagonal,
	\item $\lambda_{\mathcal Q}$ the spectral gap of $\mathcal Q$,
	\item $\lambda_\pi := \lim_{L \rightarrow +\infty} -\frac{1}{L} \log \pi \left( \{ x > L\}\right)$ the exponential decay rate of the invariant measure.
\end{itemize}
\end{Not}

The quantitative aspect of our results comes from non-asymptotic large deviation bounds for additive functionals of Markov processes. Our main tool will be the following lemma which is an immediate consequence of~\cite[Theorem 1]{wu2000deviation}.

\begin{Lem} \label{lem:LDP_velocities} 
If $\sigma$ has transition rates \ref{fig:instantaneous_tumble_relative_particle_velocity_transition_rates} (resp.~\ref{fig:finite_tumble_relative_particle_velocity_transition_rates}) and initial distribution \( \mu \) then for all $R > 0$
$$
\mathbb P \left(\frac{1}{t} \int_0^t \frac{\sigma(s)}{2} ds \ge R \right) \le \left\Vert \frac{\mu}{\pi_{\mathcal Q}} \right\Vert_{L^2(\pi_{\mathcal Q})} e^{-t I(R)}
$$
where $I$ is Legendre transform of
$$
\Lambda(u) = \sup_{\left\langle \mathcal S f, f \right\rangle =1} \left\langle \mathcal S (\mathcal Q + u {\tilde {\mathcal V}}) f, f \right\rangle.
$$
\end{Lem}

We can now state the main result of this section.

\begin{Thm}[Geometric ergodicity] \label{thm:long_time_behavior_linear_process}
\begin{itemize}
	\item[(i)] The instantaneous (resp.~finite) linear process has a unique invariant measure~$\pi$.
	\item[(ii)] For all $\lambda \in \left(0, I\left(\frac{c}{v}\right)\right)$ and \( (x, \sigma) \in E \), one has
	\begin{align*}
		&\tv{\delta_{(x, \sigma)} P_t - \pi} \le \\
		&\qquad \left(e^{\frac{\lambda}{2 c} x} + \frac{\lambda}{I\left(\frac{c}{v}\right) - \lambda} \frac{e^{\left(\frac{I'\left(\frac{c}{v}\right)}{2v} - \frac{I\left(\frac{c}{v}\right) - \lambda}{2c}\right) x}}{\sqrt{\pi_{\mathcal Q}\left(\{\sigma\}\right)}} + 2 + \pi(f) + 2t \left( \max_{\sigma \in \Sigma} \mathcal Lf(0, \sigma)\right)\right) e^{-\frac{\lambda_{\mathcal Q} \lambda}{\lambda_{\mathcal Q} + \lambda}t}
	\end{align*}
	where $I'\left(\frac{c}{v}\right)$ is any subderivative, \( \mathcal L \) is the generator of the process and
	\[
	f(x, \sigma) = \mathbb E_{(x, \sigma)} \left[e^{\lambda \tau_0}\right] \text{ with } \tau_0 =  \inf\{ t > 0 : x(t) = 0 \}.
	\]
\end{itemize}
\end{Thm}

\begin{Rem}
In the instantaneous linear case $I\left(R\right) = 2\omega\left(1 - \sqrt{1 - R^2}\right)$ so $I'\left(\frac{c}{v}\right)$ exists in the usual sense. We state Theorem~\ref{thm:long_time_behavior_linear_process} in terms of subderivatives of the convex function $I$ to emphasize the fact that it is not necessary to assume that it is differentiable, which might not be the case if we considered transition rates other than \ref{fig:instantaneous_tumble_relative_particle_velocity_transition_rates} and \ref{fig:finite_tumble_relative_particle_velocity_transition_rates}.
\end{Rem}

\begin{Rem} 
The prefactor in Theorem~\ref{thm:long_time_behavior_linear_process} (ii) has an exponential $x$-dependence. We would expect the same $x$-dependence if we turned to Foster-Lyapunov techniques~\cite{meyn93,hairer11}.
\end{Rem}

The following lemma, the proof of which is delayed until the end of the section, makes the exponential decay rates explicit in terms of model parameters.

\begin{Prop}[Decay rates as a function of model parameters]\label{prop:explicit_decay_rates} \mbox{}
	\begin{itemize}
		\item[(i)] For the instantaneous linear process $\lambda_{\mathcal Q} = 2\omega$ and $I\left(R\right) = 2\omega \left(1 - \sqrt{1 - R^2}\right)$.
		\item[(ii)] For the finite linear process $\lambda_{\mathcal Q} = \alpha$ and
		$$
		I\left(R\right) \ge \frac{3 - \sqrt{5}}{4} \min\left( \alpha R, \alpha \left(1 + \frac{\alpha}{\beta}\right) R^2 \right).
		$$
	\end{itemize}
\end{Prop}

Combining Theorem~\ref{thm:long_time_behavior_linear_process} and Proposition~\ref{prop:explicit_decay_rates} and using the inequality $\frac{a b}{a + b} \ge \frac{1}{2} \min(a, b)$ for $a, b > 0$ yields
$$
\varliminf -\frac{1}{t} \log \tv{\delta_{(x, \sigma)} P_t - \pi} \ge \omega \left(1 - \sqrt{1 - \frac{c^2}{v^2}}\right) \ge \frac{\omega}{2} \frac{c^2}{v^2}
$$
in the instantaneous case and
$$
\varliminf -\frac{1}{t} \log \tv{\delta_{(x, \sigma)} P_t - \pi} \ge \frac{3 - \sqrt{5}}{8} \min \left( \alpha \frac{c}{v}, \alpha \left(1 + \frac{\alpha}{\beta}\right) \frac{c^2}{v^2} \right)
$$
in the finite case, which are the lower bounds from Theorem~\ref{thm:convergence_of_all_processes}. The rest of the section is dedicated to proving Theorem~\ref{thm:long_time_behavior_linear_process} and Proposition~\ref{prop:explicit_decay_rates}. We start by deducing quantitative bounds for $\mathbb E_{(x, \sigma)} \left[e^{\lambda \tau_0}\right]$ from Lemma~\ref{lem:LDP_velocities}.

\begin{Lem}[Lyapunov function] \label{lem:lyapunov_function} For all $\lambda \in \left( 0, I \left( \frac{c}{v} \right) \right)$ we have
$$
\mathbb E_{(x, \sigma)} \left[e^{\lambda {\tau_0}}\right] \le e^{\frac{\lambda}{2 c} x} + \frac{\lambda}{I\left(\frac{c}{v}\right) - \lambda} \frac{e^{\left(\frac{I'\left(\frac{c}{v}\right)}{2v} - \frac{I\left(\frac{c}{v}\right) - \lambda}{2c}\right) x}}{\sqrt{\pi_{\mathcal Q}\left(\{\sigma\}\right)}}
$$
where ${\tau_0} = \inf \{t > 0: x(t) = 0\}$ and $I'\left( \frac{c}{v} \right)$ is any subderivative.
\end{Lem}

\begin{proof}
We have
$$
{\tau_0} > t \implies x(0) - 2ct + v \int_0^t \sigma(s) ds > 0 \implies \frac{1}{t} \int_0^t \frac{\sigma(s)}{2} ds > \frac{c}{v} - \frac{x(0)}{2vt},
$$
so that
\begin{align*}
\mathbb E_{(x, \sigma)} \left[ e^{\lambda {\tau_0}} \right] &= 1 + \int_0^{+\infty} \lambda e^{\lambda t} \mathbb{P}_{(x, \sigma)} \left( {\tau_0} > t \right) dt \\
&\le 1 + \int_0^{\frac{x}{2c}} \lambda e^{\lambda t} dt + \int_{\frac{x}{2c}}^{+\infty} \lambda e^{\lambda t} \mathbb{P}_{(x,\sigma)} \left( \frac{1}{t} \int_0^t \frac{\sigma(s)}{2} ds \ge \frac{c}{v} - \frac{x}{2vt}\right) dt \\
&\le e^{\frac{\lambda}{2 c} x} + \int_{\frac{x}{2c}}^{+\infty} \lambda e^{\lambda t} \left\Vert \frac{\delta_\sigma}{\pi_{\mathcal Q}}\right\Vert_{L^2(\pi_{\mathcal Q})} e^{-I \left(\frac{c}{v} - \frac{x}{2vt}\right) t}dt
\end{align*}
using Lemma~\ref{lem:LDP_velocities} for the last inequality. Because $I'\left(\frac{c}{v}\right)$ is a subderivative
$$
I \left( \frac{c}{v} - \frac{x}{2 vt}\right) \ge I \left(\frac{c}{v}\right) - I'\left(\frac{c}{v}\right) \frac{x}{2vt} \iff -t I \left(\frac{c}{v} - \frac{x}{vt}\right) \le -t I\left(\frac{c}{v}\right) + I'\left(\frac{c}{v}\right) \frac{x}{2v}
$$
so the desired result follows from $\left\Vert \frac{\delta_\sigma}{\pi_{\mathcal Q}}\right\Vert_{L^2(\pi_{\mathcal Q})} = \frac1{\sqrt{\pi_{\mathcal Q}\left(\{\sigma\}\right)}}$. 
\end{proof}

The following lemma compares the exponential rates appearing in Lemma~\ref{lem:lyapunov_function}. In particular, it ensures that $(x, \sigma) \mapsto \mathbb E_{(x, \sigma)} \left[ e^{\lambda \tau_0}\right]$ is $\pi$-integrable. This is also a consequence of~\cite[Theorem 4.3]{meyn93}.

\begin{Lem}\label{lem:lyapunov_function_integrability}
We have
$$
\frac{I\left(\frac{c}{v}\right)}{2c} \le \frac{I'\left(\frac{c}{v}\right)}{2v} \le \lambda_\pi
$$
where $I'\left(\frac{c}{v}\right)$ is any subderivative.
\end{Lem}

\begin{proof}
By definition of $I'\left(\frac{c}{v}\right)$ we have $I(R) \ge I\left(\frac{c}{v}\right) + \left(R - \frac{c}{v}\right) I'\left(\frac{c}{v}\right)$ for all $R \in \mathbb R$. Because $I(0) = 0$ this is equivalent to the first inequality when $R=0$. \\

We now turn in to the second inequality, which we prove in the instantaneous case. The finite case can be treated using similar arguments. If we set $\tau_{(0, 2)} := \inf \{t > 0 : (x(t), \sigma(t)) = (0, 2)\}$ then we have the representation
$$
\pi\left( A \right) = \frac{\mathbb E_{(0, 2)} \left[ \int_0^{\tau_{(0, 2)}} 1_{\left\{X(s) \in A\right\}} dt \right]}{\mathbb E_{(0, 2)}\left[\tau_{(0, 2)}\right]}.
$$
for all measurable sets $A$ (see~\cite[Theorem 2.1]{khasminskii60} and its proof). In particular
$$
\pi\left(\{ x > a \}\right) = \frac{\mathbb E_{(0, 2)} \left[ \int_0^{\tau_{(0, 2)}} 1_{\left\{x(t)>a\right\}} dt \right]}{\mathbb E_{(0, 2)}\left[\tau_{(0, 2)}\right]}.
$$
So using
$$
x(t) > a \text{ and } t < \tau_{(0, 2)} \implies -2ct + v \int_0^t \sigma(s) ds > a \implies \frac{1}{t} \int_0^t \frac{\sigma(s)}{2} ds > \frac{c}{v} + \frac{a}{2tv}
$$
and Lemma~\ref{lem:LDP_velocities} we deduce
\begin{align*}
\mathbb E_{(0, 2)} \left[ \int_0^{\tau_{(0, 2)}} 1_{\left\{x(t)>a\right\}} dt \right] &= \int_0^{+\infty} \mathbb P_{(0, 2)} \left( x(t) > a \text{ and } t < \tau_{(0, 2)} \right)dt \\
&\le \left\Vert \frac{\delta_2}{\pi_{\mathcal Q}}\right\Vert_{L^2(\pi_{\mathcal Q})} \int_0^{+\infty} e^{-t I\left(\frac{c}{v} + \frac{a}{2tv}\right)} dt.
\end{align*}

Because $I'\left(\frac{c}{v}\right)$ is a subderivative
$$
I\left(\frac{c}{v}+ \frac{a}{2tv}\right) \ge I\left(\frac{c}{v}\right) + I'\left(\frac{c}{v}\right) \frac{a}{2tv} \implies -t I\left(\frac{c}{v}\right) - I'\left(\frac{c}{v}\right) \frac{a}{2v} \ge -t I\left(\frac{c}{v}+ \frac{a}{2tv}\right)
$$
and thus
$$
\int_0^{+\infty} e^{-t I\left(\frac{c}{v} + \frac{a}{2tv}\right)} dt \le \frac{e^{-\frac{I'\left(\frac{c}{v}\right)}{2v} a}}{I\left(\frac{c}{v}\right)}.
$$

Hence the decay rate of $\pi\left(\{x > a\}\right)$ is at least $\frac{I'\left(\frac{c}{v}\right)}{2 v}$ which is precisely the second inequality.
\end{proof}

The following final preliminary lemma constructs the synchronous coupling and shows that it has two important properties: after a short period of time the velocities of the two copies \( \sigma(t) \) and \( \tilde \sigma(t) \) coincide and the order between $x$ and $\tilde x$ is preserved.

\begin{Lem}[Synchronous coupling] \label{lem:linear_synchronous_coupling}
For all $(\sigma_0, x_0), (\tilde x_0, \tilde \sigma_0) \in E$ there exists a coupling $\mathbb Q_{(\sigma_0, x_0), (\tilde x_0, \tilde \sigma_0)}$ of $X(t) = (x(t), \sigma(t))$ and $\tilde X(t) = (\tilde x(t), \tilde \sigma(t))$ such that
\begin{itemize}
\item[(i)] $X$ and $\tilde X$ are instantaneous (resp.~finite) linear processes with initial states $(x_0, \sigma_0)$ and $(\tilde x_0, \tilde \sigma_0)$,
\item[(ii)] $\mathbb P \left( \sigma(s) = \tilde \sigma(s) \text{ for all } s \ge t\right) \ge 1 -  2 e^{-\lambda_{\mathcal Q} t}$ where $\lambda_{\mathcal Q}$ is the spectral gap of $\mathcal Q$,
\item[(iii)] if $\sigma(s) = \tilde \sigma(s)$ for all $s \ge t$ then
\begin{align*}
x(t) \le \tilde x(t) \implies x(s) \le \tilde x(s) &\text{ for all } s \ge t, \\
\tilde x(t) \le x(t) \implies \tilde x(s) \le  x(s) &\text{ for all } s \ge t,
\end{align*}
meaning that the ordering of $x$ and $\tilde x$ is preserved.
\end{itemize}
\end{Lem}

\begin{proof} (i) Let $f : \{\pm 1\}^2 \rightarrow \Sigma$ (resp.~$f : \{0, \pm 1\}^2 \rightarrow \Sigma$) be defined by $f(\sigma_1, \sigma_2) = \sigma_2 - \sigma_1$ (resp. $f(\sigma_1, \sigma_2) = \sigma_2 - \sigma_1$ when $\sigma_1\ne \sigma_2$ as well as $f(\pm 1, \pm 1) = 0_\pm$ and $f(0, 0) = 0_0$) and choose $\sigma_{1, 0}, \sigma_{2, 0}, \tilde \sigma_{1, 0}, \tilde \sigma_{2, 0}$ such that $\sigma_0 = f(\sigma_{1, 0}, \sigma_{2, 0})$ and $\tilde \sigma_0 = f(\tilde \sigma_{1, 0}, \tilde \sigma_{2, 0})$. 
	
\begin{figure}[H]
\centering

\begin{subfigure}[t]{0.3\textwidth}
\centering
\scalebox{0.8}{
\begin{tikzpicture}[-latex, node distance=2cm, main/.style = {draw, circle, minimum size=.8cm, font=\footnotesize}]
\node[main, minimum size = 30pt] (++) {$++$};
\node[main, minimum size = 30pt] [right of=++] (+-) {$+-$};

\node[main, minimum size = 30pt] [below of=++] (-+) {$-+$};
\node[main, minimum size = 30pt] [right of=-+] (--) {$--$};

\path[every node/.style={font=\footnotesize}]
(++) edge[bend right=20] node[fill=white] {$\omega$} (--)
(--) edge[bend right=20] node[fill=white] {$\omega$} (++)

(+-) edge node[fill=white] {$\omega$} (++)
(+-) edge node[fill=white] {$\omega$} (--)
(-+) edge node[fill=white] {$\omega$} (++)
(-+) edge node[fill=white] {$\omega$} (--)
;
\end{tikzpicture}
}
\caption{Instantaneous tumble}
\label{fig:instantaneous_single_velocity_coupling}
\end{subfigure}
\qquad
\begin{subfigure}[t]{0.3\textwidth}
	\centering
	\scalebox{0.8}{
	\begin{tikzpicture}[-latex, node distance=2.2cm, main/.style = {draw, circle, minimum size=.8cm, font=\footnotesize}]
	\node[main, minimum size = 30pt] (++) {$++$};
	\node[main, minimum size = 30pt] [right of=++] (+0) {$+0$};
	\node[main, minimum size = 30pt] [right of=+0] (+-) {$+-$};
	
	\node[main, minimum size = 30pt] [below of=++] (0+) {$0+$};
	\node[main, minimum size = 30pt] [right of=0+] (00) {$00$};
	\node[main, minimum size = 30pt] [right of=00] (0-) {$0-$};
	
	\node[main, minimum size = 30pt] [below of=0+] (-+) {$-+$};
	\node[main, minimum size = 30pt] [right of=-+] (-0) {$-0$};
	\node[main, minimum size = 30pt] [right of=-0] (--) {$--$};
	
	\path[every node/.style={font=\footnotesize}]

	(++) edge[bend right=20] node[fill=white] {$\alpha$} (00)
	(--) edge[bend right=20] node[fill=white] {$\alpha$} (00)
	(00) edge[bend right=20] node[fill=white] {$\beta/2$} (++)
	(00) edge[bend right=20] node[fill=white] {$\beta/2$} (--)
	
	(+-) edge node[fill=white] {$\alpha$} (00)
	(-+) edge node[fill=white] {$\alpha$} (00)
	(+0) edge node[fill=white] {$\alpha$} (00)
	(-0) edge node[fill=white] {$\alpha$} (00)
	(0+) edge node[fill=white] {$\alpha$} (00)
	(0-) edge node[fill=white] {$\alpha$} (00)
	
	(0-) edge node[fill=white] {$\beta/2$} (+-)
	(0-) edge node[fill=white] {$\beta/2$} (--)
	
	(0+) edge node[fill=white] {$\beta/2$} (++)
	(0+) edge node[fill=white] {$\beta/2$} (-+)
	
	(+0) edge node[fill=white] {$\beta/2$} (+-)
	(+0) edge node[fill=white] {$\beta/2$} (++)
	
	(-0) edge node[fill=white] {$\beta/2$} (--)
	(-0) edge node[fill=white] {$\beta/2$} (-+)
	;
	\end{tikzpicture}
	}
	\caption{Finite tumble}
	\label{fig:finite_single_velocity_coupling}
\end{subfigure}
\caption{Single velocity coupling}
\label{fig:single_velocity_coupling}
\end{figure}
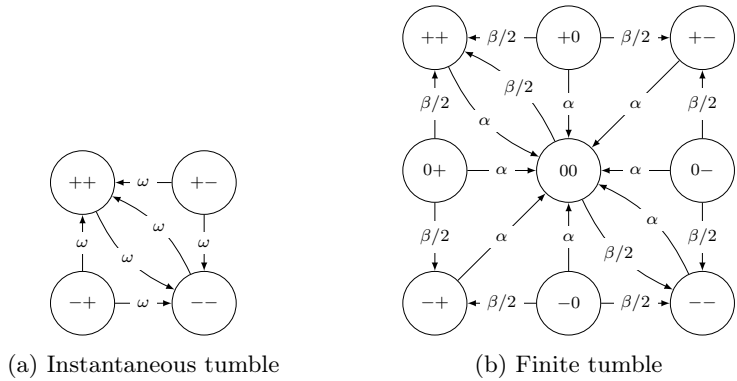

Let $(\sigma_1, \tilde \sigma_1)$ and $(\sigma_2, \tilde \sigma_2)$ be independent Markov jump processes with the transition rates of Figure~\ref{fig:instantaneous_single_velocity_coupling} (resp.~\ref{fig:finite_single_velocity_coupling}) and initial states $(\sigma_{1, 0}, \tilde \sigma_{1, 0})$ and $(\sigma_{2, 0}, \tilde \sigma_{2, 0})$. It follows from~\cite[Theorem 2.4]{ball93} that $\sigma = f(\sigma_1, \sigma_2)$ and $\tilde \sigma = f(\tilde \sigma_1, \tilde \sigma_2)$ are Markov jump processes. By~\cite[Theorem 2.3]{ball93}, their transition rates are given by Figure~\ref{fig:instantaneous_tumble_relative_particle_velocity_transition_rates} (resp.~\ref{fig:finite_tumble_relative_particle_velocity_transition_rates}). \\

Finally, recursively construct
\begin{align*}
x(t) &= \max(0, x_0 + (v \sigma_0 - 2c) t) \text{ for } t \in [0, T_1),\\
x(t) &= \max(0, x(T_1) + (v \sigma(T_1) - 2c) (t-T_1)) \text{ for } t \in [T_1, T_2),
\end{align*}
and so forth. Construct $\tilde x(t)$ similarly and set
$$
X(t) = (x(t), \sigma(t)), \quad \tilde X(t) = (\tilde x(t), \tilde \sigma(t)).
$$

(ii) Recall that $\lambda_{\mathcal Q} = 2\omega$ in the instantaneous case and $\lambda_{\mathcal Q} = \alpha$ in the finite case by Proposition~\ref{prop:explicit_decay_rates}. Set $\tau_{\sigma}^i = \inf \{ t \ge 0 : \sigma_i(t) = \tilde \sigma_i(t) \}$ and observe that $\sigma(s) = \tilde \sigma(s)$ for all $s \ge \max \{\tau_{\sigma}^1, \tau_{\sigma}^2\}$ hence it is enough to show $\mathbb P \left( \tau_{\sigma}^i > t\right) \le e^{-\lambda_{\mathcal Q} t}$.

In the instantaneous case, if $\sigma_{i, 0} \neq \tilde \sigma_{i, 0}$ then we have $\sigma_i = \tilde \sigma_i$ after the first jump of $(\sigma_i, \tilde \sigma_i)$ which is exponentially distributed with parameter $2 \omega$ (see figure \ref{fig:instantaneous_single_velocity_coupling}). Hence $\mathbb P \left( \tau_{\sigma}^i > t\right) \le e^{-2 \omega t}$.

In the finite case, notice that all states other than $(0, 0)$ have a transition with rate $\alpha$ to the state~$(0, 0)$ (see figure \ref{fig:finite_single_velocity_coupling}). Hence, if $(\sigma_{i, 0}, \tilde \sigma_{i, 0}) \neq (0, 0)$ then $\tau^i_{00} := \inf \{ t \ge 0 : \sigma_i(t) = \tilde \sigma_i(t) = 0\}$ is exponentially distributed with parameter $\alpha$. Hence $\mathbb P \left( \tau_{\sigma}^i > t\right) \le \mathbb P \left( \tau^i_{00} > t \right) \le e^{-\alpha t}$.

(iii) Follows from the recursive construction of $x$ and $\tilde x$ and the fact that
$$
x \mapsto \max(0, x + (v \sigma(T_k) - 2c)(t - T_k))
$$
is always a nondecreasing function.
\end{proof}

\begin{Rem}
The coupling \( (X, \tilde X) \) constructed in Lemma \ref{lem:linear_synchronous_coupling} is not jointly Markovian. Indeed, if $(\sigma_i, \tilde \sigma_i)$ has the transition rates \ref{fig:finite_single_velocity_coupling} then, even though $\sigma = f(\sigma_1, \sigma_2)$ and $\tilde \sigma = f(\tilde \sigma_1, \tilde \sigma_2)$ are Markovian, the joint process $(\sigma, \tilde \sigma)$ is not Markovian. This issue could be addressed by directly coupling the velocities \( \sigma \) and \( \tilde{\sigma} \). However, this approach would require describing a Markov jump process with \( 36 \) states in the finite linear case, which is unlikely to provide valuable insight.
\end{Rem}

We can now turn to the proof of the main result of this section.

\begin{proof}[Proof of Theorem~\ref{thm:long_time_behavior_linear_process}] (i) In the instantaneous case, Lemma~\ref{lem:lyapunov_function} and~\cite[Theorem 4.1]{meyn93} applied to the closed petite set $C = \{(0, 2)\}$, the test function $f(x,\sigma) = 1$ and the shift $\delta = 1$ imply the existence of a unique invariant probability measure. A similar argument applies to the finite case.
	
	(ii) Let $(x_0, \sigma_0) \in E$ be arbitrary but fixed and set
	$$
	\mathbb P = \int \mathbb Q_{(x_0, \sigma_0), (\tilde x, \tilde \sigma)} d\pi(\tilde x, \tilde \sigma)
	$$
	where $\mathbb Q_{(x_0, \sigma_0), (\tilde x, \tilde \sigma)}$ is the coupling constructed in Lemma~\ref{lem:linear_synchronous_coupling}. We denote $X(t) = (x(t), \sigma(t))$ (resp.~$\tilde X(t) = (\tilde x(t), \tilde \sigma(t))$) the first (resp.~second) marginal of $\mathbb P$ and observe that this is a coupling of two instantaneous (resp.~finite) linear processes with initial distributions $\delta_{(x_0, \sigma_0)}$ and $\pi$. \\
	
	Fix $\lambda \in \left(0,I\left(\frac{c}{v}\right)\right)$ and set $\theta = \frac{ \lambda}{\lambda_{\mathcal Q} + \lambda}$. For any time $T \in \mathbb R_+$ define $\tau_0^T := \{ t > T : x(t) = 0\}$ and $\tilde \tau_0^T := \{ t > T : \tilde x(t) = 0\}$. If
	\begin{align*}
	\sigma(s) = \tilde \sigma(s) \text{ for all } s \ge T, \quad x(T) \le \tilde x(T), \quad \tilde x(s_0) = 0 \text{ for some } s_0 \ge T,
	\end{align*}
then, because the ordering is preserved,
	$$
	0 \le x(s_0) \le \tilde x(s_0) = 0 \implies x(s_0) = \tilde x(s_0)
	$$
	and similarly if $\tilde x(T) \le x(T)$ and $x(s_0) = 0$ for some $s_0 \ge T$. This implies
	\begin{align*}
	\tv{\delta_{(x_0, \sigma_0)} P_t - \pi} &\le \mathbb P\left( \left\{ \sigma(s) = \tilde \sigma(s) \text{ for all } s\ge \theta t \right\}^c\right) + \mathbb P\left( \tau_0^{\theta t} > t\right) + \mathbb P\left( \tilde \tau_0^{\theta t} > t\right) \\
	&\le 2 e^{- \frac{\lambda_{\mathcal Q} \lambda}{\lambda_{\mathcal Q} + \lambda} t} + \mathbb P\left( \tau_0^{\theta t} > t\right) + \mathbb P\left( \tilde \tau_0^{\theta t} > t\right)
	\end{align*}
	using Lemma \ref{lem:linear_synchronous_coupling} for the second inequality. The Markov property yields
	\begin{align*}
	\mathbb P\left( \tau_0^{\theta t} > t\right) = \mathbb E_{(x_0, \sigma_0)}\left[ \mathbb P_{X(\theta t)} \left(\tau_0 > (1 - \theta)t\right) \right] \le \mathbb E_{(x_0, \sigma_0)}\left[ \frac{\mathbb E_{X(\theta t)} \left[e^{\lambda \tau_0}\right]}{e^{\lambda (1-\theta) t}} \right] = e^{-\frac{\lambda_{\mathcal Q} \lambda}{\lambda_{\mathcal Q} +\lambda}t} \mathbb E_{(x_0, \sigma_0)} \left[ P_{\theta t} f\right]
	\end{align*}
	where $f(x,\sigma) = \mathbb E_{(x, \sigma)} \left[ e^{\lambda \tau_0}\right]$ and \( \tau_0 = \inf \{ t > 0 : x(t) = 0 \} \). For all \( (x, \sigma) \in (0, +\infty) \times \Sigma \), we have
	$$
	\mathcal L f(x, \sigma) + \lambda f(x, \sigma) = 0 \implies \mathcal L f(x, \sigma) = -\lambda f(x, \sigma) \le 0.
	$$
	Hence $\mathcal L f$ is upper bounded and we have $P_t f = f + \int_0^t P_s \mathcal L f ds \le f + C t$ for
	\begin{align*}
		C = \max_{\sigma \in \Sigma} \mathcal L f(0, \sigma).
	\end{align*}
	Combining this with the previous inequality we get
	$$
	\mathbb P \left( \tau_0^{\theta t} > t\right) \le \left( f(x, \sigma) + Ct  \right) e^{-\frac{\lambda_{\mathcal Q} \lambda}{\lambda_{\mathcal Q} +\lambda}t},
	$$
	and the same arguments show
	$$
	\mathbb P \left( \tilde \tau_0^{\theta t} > t\right) \le \left( \pi(f) + Ct\right) e^{-\frac{\lambda_{\mathcal Q} \lambda}{\lambda_{\mathcal Q} +\lambda}t}.
	$$
	
	The result now follows from replacing $f(x, \sigma)$ by the upper bound in Lemma~\ref{lem:lyapunov_function} and combining Lemma~\ref{lem:lyapunov_function} and Lemma~\ref{lem:lyapunov_function_integrability} to see $\pi(f) < +\infty$.
\end{proof}

\begin{proof}[Proof of Proposition~\ref{prop:explicit_decay_rates}] (i) Recall that in the instantaneous case, we have
$$
\mathcal Q = \begin{pmatrix} -2\omega & 2\omega & 0 \\
\omega & -2\omega & \omega \\
0 & 2\omega & -2\omega
\end{pmatrix}, \quad
{\tilde {\mathcal V}} =
\begin{pmatrix}
1 & 0 & 0 \\
0 & 0 & 0 \\
0 & 0 & -1 \\
\end{pmatrix}, \quad
\mathcal S =
\begin{pmatrix}
\frac{1}{4} & 0 & 0 \\
0 & \frac{1}{2} & 0 \\
0 & 0 & \frac{1}{4}
\end{pmatrix}.
$$
The eigenvalues of \( \mathcal Q \) are \( 0, -2\omega \) and \( -4\omega \) so \( \lambda_\mathcal Q = 2\omega \). Notice that $\mathcal A := \mathcal Q + u {\tilde {\mathcal V}}$ is symmetric w.r.t.~the scalar product $\left\langle \mathcal S \cdot, \cdot \right\rangle$ (i.e.~$\mathcal A = \mathcal S^{-1} \mathcal A^t \mathcal S$) so that
$$
\Lambda(u) = \sup_{\left\langle\mathcal S f, f \right\rangle = 1} \left\langle \mathcal S \mathcal A f, f\right\rangle = - 2\omega + \sqrt{u^2 + 4 \omega^2}
$$
is the largest eigenvalue of $\mathcal A$ and the result follows by taking the Legendre transform.

(ii) The eigenvalues of \( \mathcal Q \) are
\[
	-2 {\alpha} - 2 {\beta},\qquad -2 {\alpha} - {\beta},\qquad -{\alpha} - {\beta},\qquad -2 {\alpha},\qquad -{\alpha},\qquad 0,
\]
so \( \lambda_{\mathcal Q} = \alpha \). By~\cite[Theorem 1.1]{lezaud01} and~\cite[Remark 1.2]{lezaud01}
$$
I(R) \ge \alpha R \frac{2 (1+r) R}{\left(1 + \sqrt{1 + 4(1+r)R}\right)^2} = \alpha R \times g((1+r)R)
$$
where $r = \frac{\alpha}{\beta}$ and \( g(X) = {2X}/{{\left(1 + \sqrt{1 + 4X}\right)}^2} \). One has
\begin{align*}
	g'(X) = \frac{2}{\sqrt{4 X + 1} {\left(\sqrt{4 X + 1} + 1\right)}^{2}} \ge 0 \quad \text{ and } \quad \left( \frac{g(X)}X\right)' = -\frac{8}{\sqrt{4 X + 1} {\left(\sqrt{4 X + 1} + 1\right)}^{3}} \le 0,
\end{align*}
so that
\[
	\inf_{X \ge 1} g(X) = \inf_{X \le 1} \frac{g(X)}X = g(1) = \frac{3 - \sqrt{5}}{4}
\]
and the result follows.
\end{proof}

\subsection{Instantaneous harmonic process}

In the harmonic case, the underlying deterministic dynamics are contracting in the sense that for any fixed $\sigma_0 \in \Sigma$ we have
$$
\partial_t x = -2\mu x + v \sigma_0 \text{ and } \partial_t \tilde x= -2\mu \tilde x + v \sigma_0 \implies x(t) - \tilde x(t) = e^{-2\mu t} (x(0) - \tilde x(0)).
$$
The random switching of $\sigma$ and the jamming interactions preserve this contractivity and hence the instantaneous harmonic process displays exponential convergence towards its invariant measure in a Wasserstein-type distance.

The same kind of contractivity is exploited in~\cite{benaim12} to obtain quantitative bounds for the convergence of processes randomly switching between contracting deterministic  dynamics. These results can also be extended to random switching between diffusions with contracting drifts (see~\cite[Section 4.1]{cloez15}). The following distance, which combines the Wasserstein distance of the $x$-marginal and the total variation distance for the $\sigma$-marginal, is naturally contracting for the instantaneous harmonic process.

\begin{Def}[\text{Mixed distance~\cite[Definition 1.4]{benaim12}}]\label{def:mixed_distance} For all probabilities $\eta, \tilde \eta$ on $[0, +\infty) \times \Sigma$ such that the $x$-marginal of $\eta$ and $\tilde \eta$ have a finite moment of order $p$, we define the mixed distance
$$
\mathcal W_p\left(\eta, \tilde \eta\right) := \inf \left\{ \left(\mathbb E \left|x - \tilde x\right|^p\right)^{\frac{1}{p}} + \mathbb P\left( \sigma \neq \tilde \sigma \right) : (x, \sigma) \sim \eta \text{ and } (\tilde x, \tilde \sigma) \sim \tilde \eta\right\}.
$$
\end{Def}

Adapting~\cite[Theorem 1.10]{benaim12} and~\cite[Corollary 1.11]{benaim12} to our setting yields the following theorem.

\begin{Thm}[Invariant measure and convergence of the instantaneous harmonic process]\label{thm:long_time_behavior_harmonic_process_wasserstein}	\mbox{}
\begin{itemize}
	\item[(i)] For all $1 \le p < q$ we have
	\begin{align*}
			&\mathcal W_p \left(\eta P_t, \tilde \eta P_t\right) \le \left( 1 + 2^\frac1s\right)^\frac1p \left( \mathcal F_q(\eta) + \mathcal F_q(\tilde \eta) \right) \exp\left({- \frac{2\omega\mu/s}{\omega/s + p\mu} t}\right) + 2 e^{-2\omega t}
	\end{align*}
	where \( s > 1 \) is such that \( \frac{p}{q} + \frac1s = 1 \) and 
	\[
		\mathcal F_r(\nu) := \left(\int \max(x, v/\mu)^r d\nu(x, \sigma)\right)^\frac1r.
	\]
	\item[(ii)] The instantaneous harmonic process has a unique invariant probability $\pi$ and
	$$
	\mathcal W_p \left(\eta P_t, \pi\right) \le \left( 1 + 2^\frac1s\right)^\frac1p \left( \mathcal F_q(\eta) + v/\mu \right) \exp\left({- \frac{2\omega\mu/s}{\omega/s + p\mu} t}\right) + 2 e^{-2\omega t}.
	$$
\end{itemize}
\end{Thm}

The theorem is stated for the instantaneous harmonic process but can be extended to other tumbling mechanisms (for example \ref{fig:finite_tumble_relative_particle_velocity_transition_rates}) and arbitrary contracting potentials, including higher-dimensional processes such as~\cite{smith22}. Our setting allows for a somewhat simplified proof, which we include for the sake of completeness. Following~\cite{benaim12}, we again turn to a coupling argument and start by adapting Lemma~\ref{lem:linear_synchronous_coupling} to the harmonic setting.

\begin{Lem}[Synchronous coupling] \label{lem:harmonic_synchronous_coupling}
	For all $(x_0, \sigma_0), (\tilde x_0, \tilde \sigma_0) \in E$ there exists a coupling $\mathbb Q_{(x_0, \sigma_0), (\tilde x_0, \tilde \sigma_0)}$ of $X(t) = (x(t), \sigma(t))$ and $\tilde X(t) = (\tilde x(t), \tilde \sigma(t))$ such that
	\begin{itemize}
		\item[(i)] $X$ and $\tilde X$ are instantaneous harmonic processes with initial states $(x_0, \sigma_0)$ and $(\tilde x_0, \tilde \sigma_0)$,
		\item[(ii)] $\mathbb P\left(\sigma(s) = \tilde \sigma(s) \text{ for all } s \ge t\right) \ge 1 - 2 e^{-2\omega t}$,
		\item[(iii)] if $\sigma(s) = \tilde \sigma(s)$ for all $s \ge t$ then
		\begin{align*}
		|x(s) - \tilde x(s)| \le e^{-2\mu (s - t)} |x(t) -\tilde x(t)| \text{ for all } s \ge t.
		\end{align*}
	\end{itemize}
\end{Lem}

\begin{proof} We couple the velocities $\sigma, \tilde \sigma$ like in the proof of Lemma \ref{lem:linear_synchronous_coupling}. For $\sigma = 2, 0, -2$ define $x^*_\sigma = \frac{v \sigma}{2 \mu}$ and recursively construct $x(t)$ by setting
\begin{align*}
x(t) &= \max\left(0, x^*_{\sigma_0} + e^{-2\mu t}\left(x_0 - x^*_{\sigma_0}\right)\right) \text{ for } t \in [0, T_1),\\
x(t) &= \max\left(0, x^*_{\sigma(T_1+)} + e^{-2\mu(t - T_1)}\left(x(T_1) - x^*_{\sigma(T_1+)}\right)\right) \text{ for } t \in [T_1, T_2),
\end{align*}
and so forth, where $0 = T_0 < T_1 < \cdots$ are the combined jump times of $\sigma$ and $\tilde \sigma$. We define $\tilde x(t)$ similarly and finally set $X(t) = (x(t), \sigma(t))$ and $\tilde X(t) = (\tilde x(t), \tilde \sigma(t))$. Assertions (i) and (ii) follow like in the proof of Lemma \ref{lem:linear_synchronous_coupling}.

(iii) Let $t \ge 0$ be such that $\sigma(s) = \tilde \sigma(s)$ for all $s \ge t$ and let $k \in \mathbb N$ be such that $T_k \le t \le T_{k+1}$. Then we have $\sigma(T_k) = \tilde \sigma(T_k)$ and hence
\begin{align*}
x(T_{k+1}) &= \max\left(0, x^*_{\sigma(T_k+)} + e^{-2\mu(T_{k+1} - t)} \left( x(t) - x^*_{\sigma(T_k+)}\right)\right), \\
\tilde x(T_{k+1}) &= \max\left(0, x^*_{\sigma(T_k+)} + e^{-2\mu(T_{k+1} - t)} \left( \tilde x(t) - x^*_{\sigma(T_k+)}\right)\right).
\end{align*}
So because $x \mapsto \max(0, x)$ is $1$-Lipschitz we have
\begin{align*}
\left| x(T_{k+1}) - \tilde x(T_{k+1})\right| &\le \Big| \left(x^*_{\sigma(T_k+)} + e^{-2\mu(T_{k+1} - t)} \left( x(t) - x^*_{\sigma(T_k+)}\right)\right) \\
& \qquad\qquad- \left(x^*_{\sigma(T_k+)} + e^{-2\mu(T_{k+1} - t)} \left( \tilde x(t) - x^*_{\sigma(T_k+)}\right)\right)\Big|, \\
&= e^{-2\mu(T_{k+1} - t)} \left| x(t) - \tilde x(t)\right|.
\end{align*}
Iterating the same computation shows $\left| x(T_l) - \tilde x(T_l)\right| \le e^{-2\mu (T_l - t)}|x(t) - \tilde x(t)|$ for all $l \ge k+1$. Applying it once more yields $|x(s) - \tilde x(s)| \le e^{-2\mu(s - t)} |x(t) - \tilde x(t)|$ for all $s \ge t$.
\end{proof}

We can now return to the proof of the main theorem of the section.

\begin{proof}[Proof of Theorem~\ref{thm:long_time_behavior_harmonic_process_wasserstein}]
	(i) Using the coupling of Lemma~\ref{lem:harmonic_synchronous_coupling} we construct the coupling
	$$
	\int \mathbb Q_{(x, \sigma), (\tilde x, \tilde \sigma)} d(\eta \otimes \tilde \eta)(x, \sigma, \tilde x, \tilde \sigma)
	$$
	with initial distribution $\eta \otimes \tilde \eta$. Fix $\beta \in (0, 1)$ to be optimized later and define the event
	$$
	A_t = \{\sigma(s) = \tilde \sigma(s) \text{ for all } s \ge t \}.
	$$
	Notice that $t \mapsto x(t)$ is non-increasing up to time $T = \inf \{t \ge 0 : x(t) \le v/\mu\}$ and $x(t) \le v/\mu$ for all $t \ge T$. Hence $0 \le x(t) \le \max(x(0), v/\mu)$ for all $t \ge 0$. Lemma~\ref{lem:harmonic_synchronous_coupling} (iii) implies
	\begin{align*}
	\mathbb E \left[ |x(t) - \tilde x(t)|^p 1_{A_{\beta t}}\right] &\le e^{-2\mu (1 - \beta)t  p} \mathbb E \left[ |x(\beta t) - \tilde x(\beta t)|^p \right] \\
	&\le e^{-2\mu (1 - \beta)t  p} \left[ \left(\int \max(x, v/\mu)^p d\eta(x, \sigma)\right)^\frac1p + \left(\int \max(x, v/\mu)^p d\tilde\eta(x, \sigma)\right)^\frac1p  \right]^p \\
	&\le e^{-2\mu (1 - \beta)t  p} \left( \mathcal F_q(\eta) + \mathcal F_q(\tilde \eta) \right)^p
	\end{align*}
	where
	\[
		\mathcal F_r(\nu) := \left(\int \max(x, v/\mu)^r d\nu(x, \sigma)\right)^\frac1r.
	\]
	Furthermore, if $s > 1$ is such that $\frac{p}{q} + \frac{1}{s} = 1$ then by Hölder's inequality
	\begin{align*}
	\mathbb E \left[ |x(t) - \tilde x(t)|^p 1_{A^c_{\beta t}} \right] &\le \mathbb E \left[ |x(t) - \tilde x(t)|^q\right]^{\frac{p}{q}} \mathbb P\left(A^c_{\beta t}\right)^{\frac{1}{s}} \\
	&\le\left( \mathcal F_q(\eta) + \mathcal F_q(\tilde \eta) \right)^p \left(2 e^{-2 \omega {\beta t}}\right)^\frac1s.
	\end{align*}
	
	Taking $\beta = \frac{p \mu}{\omega/s + p 
		\mu}$ leads to $2\mu (1 - \beta) p = 2\omega \beta/s = \frac{2 \omega p \mu /s}{\omega/s + p \mu}$ and thus
	\begin{align*}
	\mathcal W_p \left(\eta P_t, \tilde \eta P_t\right) &\le E \left[ |x(t) - \tilde x(t)|^p \right]^\frac1p + \mathbb P\left( \sigma(t) \neq \tilde \sigma(t) \right) \\
	&\le \left( 1 + 2^\frac1s\right)^\frac1p \left( \mathcal F_q(\eta) + \mathcal F_q(\tilde \eta) \right) \exp\left({- \frac{2\omega\mu/s}{\omega/s + p\mu} t}\right) + 2 e^{-2\omega t}.
	\end{align*}
	(ii) Because an instantaneous harmonic process started inside the compact set $[0, v/\mu] \times \Sigma$ stays inside that set almost surely, the Krylov–Bogolyubov argument ensures the existence of an invariant measure. Furthermore, almost surely, the process reaches \( [0, v/\mu] \times \Sigma \) and then stays inside that set indefinitely. This implies that the support of any invariant measure is contained in \( [0, v/\mu] \times \Sigma \). In particular, \( \mathcal F_q(\pi) \le \frac v\mu < +\infty \) for any invariant measure \( \pi \) and any \( q \ge 1 \). Hence the contractivity shown in assertion (i) implies the uniqueness of \( \pi \) and the bound for $\mathcal W_p\left(\eta P_t, \pi\right)$.
\end{proof}

\begin{Rem}\label{rem:instantaneous_harmonic_asymptoic_decay_rate}
Let \( (x, \sigma) \in E \) be fixed. One has \( \mathcal F_q(\delta_{(x, \sigma)}) = x < +\infty \) for all \( q \ge 1 \) hence, by Theorem~\ref{thm:long_time_behavior_harmonic_process_wasserstein}
\[
	\varliminf_{t \rightarrow +\infty} -\frac{1}{t} \log \mathcal W_p\left(\delta_{(x, \sigma)} P_t, \pi\right) \ge \frac{2\omega \mu /s}{\omega/s + p \mu} \ge \min \left( \frac\omega{p s}, \mu \right).
\]
Taking \( q \to +\infty \) then leads to \( \frac1s = 1 - \frac pq \to 1 \) and
\[
	\varliminf_{t \rightarrow +\infty} -\frac{1}{t} \log \mathcal W_p\left(\delta_{(x, \sigma)} P_t, \pi\right) \ge \min \left( \frac\omega{p}, \mu \right).
\]
\end{Rem}

\begin{Rem}
Exponential convergence towards the invariant measure in total variation distance could be obtained using Foster-Lyapunov techniques~\cite{meyn93,hairer11}. However, we do not expect the bounds to be quantitatively useful.
\end{Rem}

\section{Optimality of the convergence rate}\label{sec:optimality}

In this section, we show that the convergence rate of Section~\ref{subsec:convergence_rates_of_instantaneous_and_finite_linear process} is optimal (resp.~optimal in the regime where \( (1 + \alpha/\beta) (c/v) \) is bounded) for the instantaneous (resp.~finite) linear process. 

\subsection{Instantaneous linear process}\label{sec:instantaneous_linear_lower_bound}

In this section, we prove that for all $(x, \sigma) \in E$
$$
\varlimsup_{t \rightarrow +\infty} -\frac{1}{t} \log \tv{\delta_{(x, \sigma)} P_t - \pi} \le 4 \omega \frac{c^2}{v^2}
$$
thus showing that the lower bound for the exponential decay rate obtained in section \ref{subsec:convergence_rates_of_instantaneous_and_finite_linear process} can only be improved by a constant factor. \\

In order to do this, we want to exploit the fact that the process takes a `long time' to explore the parts of the state space where the inter-particle distance $x$ is large. The central observation is that, for any fixed $a > 0$, because
$$
\tv{\delta_{(x,\sigma)} P_t - \pi} \ge \pi\left(\{x > a t\}\right) - \mathbb P_{(x, \sigma)} \left( x(t) > at \right)
$$
we have that
\begin{equation}\label{eq:negligibility_condition}
\varliminf_{t\rightarrow +\infty} -\frac{1}{t} \log \mathbb P_{(x, \sigma)} \left( x(t) > at \right) > a\lambda_\pi
\end{equation}
implies
$$
\varlimsup_{t \rightarrow +\infty} -\frac{1}{t} \log \tv{\delta_{(x, \sigma)} P_t - \pi} \le a \lambda_\pi
$$
where \( \lambda_\pi \) is as in Notation~\ref{not:ldp}.

The lower the constant $a$, the sharper the upper bound $a \lambda_\pi$. Hence we want to take $a > 0$ as small as possible, while also ensuring that $\mathbb P_{(x, \sigma)} \left( x(t) > at \right)$ is negligible in the sense of~\eqref{eq:negligibility_condition}.

\begin{Rem}

The velocity of the process is bounded above by \( 2(v-c) \). Hence, for all \( a > 2(v-c) \), one has
\[
	\mathbb P_{(x, \sigma)} \left( x(t) > a t \right) = 0
\]
for large enough \( t \). This yields the bound
\[
	\varlimsup_{t \rightarrow +\infty} -\frac{1}{t} \log \tv{\delta_{(x, \sigma)} P_t - \pi} \le 2(v - c) \lambda_\pi = 4 \omega \frac c{v + c},
\]
but unfortunately does not capture the true decay rate.

\end{Rem}

Computing \( \mathbb P_{(x, \sigma)} \left( x(t) > at \right) \) involves solving a system of differential equations with independent variables \( x \) and \( t \). To simplify the calculation, we turn to the Laplace transform of a hitting time, reducing the problem to a system of differential equations with \( x \) as the sole independent variable.

\begin{Lem}\label{lem:Lambda_instantaneous_tumble} Set $\tau_L := \inf \{ t \ge 0: x(t) > L\}$. For all $u < 0$ and $(x, \sigma) \in E$ we have
$$
\lim_{L \rightarrow +\infty} \frac{1}{L} \log \mathbb E_{(x, \sigma)} \left[ e^{u \tau_L}\right] = \frac{c u - 2  c {\omega} - \sqrt{u^{2} v^{2} - 4  u v^{2} {\omega} + 4  c^{2} {\omega}^{2}}}{2  {\left(v^{2} - c^{2}\right)}} =: \Lambda(u).
$$
\end{Lem}

We postpone the proof of this lemma until the end of the section.

\begin{Lem} \label{lem:instantaneous_linear_exponential_decay_rate_upper _bound}
For all $a > \frac{1}{\Lambda'(0)}$ we have
$$
\varliminf_{t \rightarrow +\infty} -\frac{1}{t} \log \mathbb P_{(x, \sigma)} \left( x(t) > a t \right) > a \lambda_\pi
$$
for all $(x, \sigma) \in E$. Furthermore
$$
\varlimsup_{t \rightarrow +\infty} -\frac{1}{t} \log \tv{\delta_{(x, \sigma)} P_t - \pi} \le 4 \omega \frac{c^2}{v^2}.
$$

\end{Lem}

\begin{proof} Let $(x, \sigma) \in E$ be given. For all $u < 0$
\begin{align*}
\mathbb P_{(x, \sigma)} \left(x(t) > a t\right) &\le \mathbb P_{(x, \sigma)} \left( \tau_{at} \le t \right) \le e^{-ut} \mathbb E_{(x, \sigma)} \left[ e^{u \tau_{a t}} \right]
\end{align*}
so that by Lemma~\ref{lem:Lambda_instantaneous_tumble} we have
$$
\varlimsup_{t \rightarrow +\infty} \frac{1}{t} \log \mathbb P_{(x, \sigma)} \left(x(t) > a t\right) \le  a \Lambda(u) - u =: G(u).
$$

Because $G(0) = -a\lambda_\pi$ if $G'(0) > 0 \iff a > \frac{1}{\Lambda'(0)}$ then
$$
a \Lambda(u) - u < -a \lambda_\pi
$$
when $u \rightarrow 0-$. Thus
$$
\frac{\lambda_\pi}{\Lambda'(0)} = 4 \omega \frac{c^2}{c^2 + v^2} \le 4 \omega \frac{c^2}{v^2}.
$$
\end{proof}

\begin{proof}[Proof of Lemma~\ref{lem:Lambda_instantaneous_tumble}] Let $\omega, c, v > 0$, $u < 0$ and $(x_0, \sigma_0) \in E$ be given. Fix $L > x_0$ and set \mbox{$f_\sigma(x) = \mathbb E_{(x, \sigma)} \left[e^{u \tau_L}\right]$} for $\sigma \in \Sigma$ and $x \in [0, L]$. Further define $F(x) = (f_2(x), f_0(x), f_{-2}(x))^t$. We know that
	$$
	\mathcal V F' + \mathcal Q F + u F =  0 \iff F' = \mathcal A F
	$$
	on $(0, L)$ where $\mathcal A = - \mathcal V^{-1} \left(\mathcal Q + u I\right)$. Together with the boundary conditions $f_2(L) = 1$ and
	$$
	\left.\begin{array}{r}
	\omega f_2(0) + \omega f_{-2}(0) - 2\omega f_0(0) + u f_0(0) = 0\\
	2\omega f_0(0) - 2\omega f_{-2}(0) + u f_{-2}(0) = 0
	\end{array}\right\} \iff
	\begin{pmatrix}
	f_2(0) \\ f_0(0) \\ f_{-2}(0)
	\end{pmatrix}
	=
	f_{-2}(0) \underbrace{\begin{pmatrix}
		1 - \frac{2  u}{\omega} + \frac{u^{2}}{2  {\omega}^{2}} \\
		1 -\frac{u}{2  {\omega}} \\
		1
		\end{pmatrix}}_{=: F_0}
	$$
	this fully determines $F$. \\
	
	We deduce
	$$
	F(x_0) = f_{-2}(0) e^{x_0 \mathcal A} F_0
	$$
	where $f_{-2}(0)$ depends on $L$ but $e^{x_0 \mathcal A} F_0$ does not. Hence $\lim_L \frac{1}{L} \log \mathbb E_{(x_0, \sigma_0)} \left[ e^{u \tau_L} \right] = \lim_L \frac{1}{L} \log f_{-2}(0)$. \\
	
	Furthermore the boundary condition $f_2(L)=1$ can be rewritten as
	$$
	f_{-2}(0) \underbrace{\left[ (1, 0, 0) e^{L \mathcal A} F_0 \right]}_{:= \Gamma(L)} = 1.
	$$
	
	The characteristic polynomial $P_{\mathcal A}$ of $\mathcal A$ satisfies $-\det(\mathcal V)  P_\mathcal A = P_1 P_2$ where
	$$
	P_1 = 2 c X - u + 2  {\omega} \text{ and } P_2 = 4  {\left(c^{2} - v^{2}\right)} X^{2} - 4  c{\left(u - 2 {\omega}\right)} X + u^{2} - 4  u {\omega}.
	$$
	$P_1 P_2$ is a cubic polynomial with discriminant
	$$
	\Delta = 256  {\left(u^{2} v^{2} - 4  u v^{2} {\omega} + 4  c^{2} {\omega}^{2} + 4  v^{2} {\omega}^{2}\right)}^{2} {\left(u^{2} v^{2} - 4  u v^{2} {\omega} + 4  c^{2} {\omega}^{2}\right)}
	$$
	so whenever $u < 0$ we have $\Delta > 0$ and $P_\mathcal A$ has three distinct real roots. Furthermore, when $u < 0$ the root of $P_1$ is negative and $P_2$ has one negative and one positive root, because its leading coefficient and constant term have opposite signs. \\
	
	Let $\lambda_1$ be the negative root of $P_1$ and $\lambda_2^+, \lambda_2^-$ be the positive and negative root of $P_2$ respectively. Because all eigenvalues of $\mathcal A$ are distinct, there exist $a_1, a_2^+, a_2^- \in \mathbb R$ that do not depend on $L$ such that
	$$
	\Gamma(L) = a_1 e^{\lambda_1 L} + a_2^+ e^{\lambda_2^+ L} + a_2^- e^{\lambda_2^- L}.
	$$
	
	Assume by contradiction that $a^+_2 < 0$. This would imply $\Gamma(L) \rightarrow -\infty$ and hence $f_{-2}(0) < 0$ for large~$L$, which is absurd. Assume by contradiction that $a^+_2 = 0$ then $|\Gamma(L)| \rightarrow 0$ and hence $|f_{-2}(0)| \rightarrow +\infty$ which is also absurd. Hence $a_2^+ > 0$ and
	\[
		\lim_L \frac{1}{L} \log \mathbb E_{(x_0, \sigma_0)} \left[ e^{u \tau_L} \right] = \lim_L \frac{1}{L} \log \left(\frac1{\Gamma(L)} \right) = -\lambda^+_2 = \frac{c u - 2  c {\omega} - \sqrt{u^{2} v^{2} - 4  u v^{2} {\omega} + 4  c^{2} {\omega}^{2}}}{2  {\left(v^{2} - c^{2}\right)}}.
	\]
\end{proof}

\subsection{Finite linear process}

In this section, we show
$$
\varlimsup_{t \rightarrow +\infty} -\frac{1}{t} \log \tv{\delta_{(\sigma, x)} P_t - \pi} \le 4 \alpha \left(1 + \frac{\alpha}{\beta}\right) \frac{c^2}{v^2} \text{ for all } (x, \sigma) \in E.
$$
As in the previous section, we want to exploit the fact that the process is slow to explore  regions of the state space where the inter-particle distance $x$ is large. We aim again to utilize the fact that~\eqref{eq:negligibility_condition} implies
\begin{align*}
	\varlimsup_t - \frac{1}{t} \log \tv{\delta_{(x, \sigma)} P_t - \pi} \le a \lambda_\pi.
\end{align*}
The following lemma is the first step to make this upper bound quantitative.

\begin{Lem} \label{lem:finite_tumble_invariant_measure_exponential_decay_rate}
We have $\lambda_\pi = -\zeta_2$ (see Notation~\ref{not:zeta2_zeta3_piQ}).
\end{Lem}

\begin{proof} If $v \le 2c$ then Proposition~\ref{prop:finite_linear_invariant_measure} implies $\pi \left( \{ x > t \} \right) = a_2 e^{\zeta_2 t}$ for some $a_2 > 0$ and the result follows immediately. If $v > 2c$ then Proposition~\ref{prop:finite_linear_invariant_measure} implies $\pi \left( \{ x > t \} \right) = a_2 e^{\zeta_2 t} + a_3 e^{\zeta_3 t}$ for some $a_2, a_3 \in \mathbb R$ so it suffices to show that $\zeta_3 < \zeta_2$. \\

It follows from the implicit function theorem that $\zeta_2$ and $\zeta_3$ are continuously differentiable functions of $(\alpha, \beta, c, v)$ (the invertibility condition in the implicit function theorem is guaranteed by the fact that $P_2$ and $P_3$ have simple roots as shown in Lemma~\ref{lem:roots_of_P2_and_P3}). Assume by contradiction that there exist two sets of parameters $\Theta_0 = (\alpha_0, \beta_0, c_0, v_0)$ and $\Theta_1 = (\alpha_1, \beta_1, c_1, v_1)$ such that $\zeta_2(\Theta_0) < \zeta_3(\Theta_0)$ and $\zeta_2(\Theta_1) > \zeta_3(\Theta_1)$. Applying the intermediate value theorem to the function $t \mapsto \zeta_2(\Theta_0 + t(\Theta_1 - \Theta_0)) - \zeta_3(\Theta_0 + t(\Theta_1 - \Theta_0))$ yields that there exists $t^* \in (0, 1)$ such that $\zeta_2(\Theta_0 + t^*(\Theta_1 - \Theta_0)) = \zeta_3(\Theta_0 + t^*(\Theta_1 - \Theta_0))$ which contradicts Lemma~\ref{lem:roots_of_P2_and_P3}. Hence the ordering of $\zeta_2$ and $\zeta_3$ is always the same and considering any particular choice of parameters shows that we always have $\zeta_3 < \zeta_2$.
\end{proof}

We now estimate $\lim_t \frac{1}{t} \log \mathbb P_{(x, \sigma)} \left( x(t) > a t \right)$ to find a small $a > 0$ such that~\eqref{eq:negligibility_condition} holds. One could try to mimic the proof of Section~\ref{sec:instantaneous_linear_lower_bound} by computing
$$
\lim_{L \rightarrow +\infty} \frac{1}{L} \log \mathbb E_{(x, \sigma)} \left[ e^{u \tau_L}\right]
$$
but the size of the systems of differential equations involved makes this challenging. To overcome this difficulty, we exploit the fact that the velocity $\sigma$ is the difference of two simpler independent Markov jump processes. This allows us to reduce the size of the systems of differential equations in our computations.

\begin{Def}[Single velocity coupling]\label{def:relative_particle_single_particle_coupling} Let $(x_0, \sigma_0) \in [0, +\infty) \times \Sigma$ be given and define
$$
f(\sigma_1, \sigma_2) = \left\{
\begin{array}{cl}
\sigma_2 - \sigma_1 &\text{if } \sigma_1 \ne \sigma_2, \\
0_\pm & \text{if } \sigma_1 = \sigma_2 = \pm 1, \\
0_0 & \text{if } \sigma_1 = \sigma_2 = 0.
\end{array}
\right.
$$

Let $\sigma_1$ and $\sigma_2$ be independent Markov jump processes with the rates of Figure~\ref{fig:finite_tumble_single_particle_velocity_transition_rates} and initial states $\sigma_{1, 0}$ and $\sigma_{2, 0}$ such that $f(\sigma_{1, 0}, \sigma_{2, 0}) = \sigma_0$. This makes $\sigma(t) = f(\sigma_1(t), \sigma_2(t))$ a Markov jump process following figure \ref{fig:finite_tumble_relative_particle_velocity_transition_rates} with initial state $\sigma_0$. \\
	
Set $x(0) = x_0$ and $y_1(0) = y_2(0) = \frac{x_0}{2}$ and recursively define
\begin{align*}
x(t) &= \max\left[0, x(T_n) + (-2c + v \sigma(T_n)) (t - T_n)\right], \\
y_1(t) &= \max\left[0, y_1(T_n) + (-c + v \sigma_1(T_n)) (t - T_n)\right], \\
y_2(t) &= \max\left[0, y_2(T_n) + (-c - v \sigma_2(T_n))(t - T_n)\right],
\end{align*}	
for $t \in [T_n, T_{n+1})$ where $0 = T_0 < T_1 < \cdots$ are the combined jump times of $\sigma_1, \sigma_2$ and $\sigma$.
\end{Def}

The process $X(t) = (x(t), \sigma(t))$ constructed in Definition~\ref{def:relative_particle_single_particle_coupling} is clearly a finite linear process and the two auxiliary processes $y_1(t)$ and $y_2(t)$ provide a useful upper bound as shown by the following lemma.

\begin{Lem}[Order preservation] If $x(t), y_1(t)$ and $y_2(t)$ are as in Definition~\ref{def:relative_particle_single_particle_coupling} then
\[
	x(t) \le y_1(t) + y_2(t)
\]
for all $t \ge 0$ almost surely.
\end{Lem}

\begin{proof} We can decompose $[0, +\infty)$ into a sequence of intervals $[t_0, t_1], [t_1, t_2], \ldots$ such that both $x$ and $y_i$ have constant derivative on $[t_k, t_{k+1}]$. By induction, it suffices to prove that \mbox{$x(t_k) \le y_1(t_k) + y_2(t_k)$} implies $x(t) \le y_1(t) + y_2(t)$ for all $t \in [t_k, t_{k+1}]$. \\
	
Recall that, for $t \in [t_k, t_{k+1}]$
\begin{itemize}
\item if $x(t_k) > 0$ then $\partial_t x(t) = -2c + v(\sigma_2(t_k) - \sigma_1(t_k))$,
\item if $x(t_k) = 0$ and $-2c + v(\sigma_2(t_k) - \sigma_1(t_k)) \ge 0$ then $\partial_t x(t) = -2c + v(\sigma_2(t_k) - \sigma_1(t_k))$,
\item if $x(t_k) = 0$ and $-2c + v(\sigma_2(t_k) - \sigma_1(t_k)) < 0$ then $\partial_t x(t) = 0$,
\end{itemize}
and similarly
\begin{itemize}
	\item if $y_i(t_k) > 0$ then $\partial_t y_i(t) = -c + v(-1)^i \sigma_i(t_k)$,
	\item if $y_i(t_k) = 0$ and $-c + v(-1)^i \sigma_i(t_k) \ge 0$ then $\partial_t y_i(t) = -c + v(-1)^i \sigma_i(t_k)$,
	\item if $y_i(t_k) = 0$ and $-c + v(-1)^i \sigma_i(t_k) < 0$ then $\partial_t y_i(t) = 0$.
\end{itemize}

\textbf{Case 1.} If $\partial_t x = -2c + \sigma_2(t_k) - \sigma_1(t_k)$ then, by the preceding remark, we necessarily have
$\partial_t x \le \partial_t y_1 + \partial_t y_2$ so $x(t_k) \le y_1(t_k) + y_2(t_k)$ implies $x(t) \le y_1(t) + y_2(t)$ for all $t \in [t_k, t_{k+1}]$. \\

\textbf{Case 2.} If $\partial_t x > -2c + \sigma_2 - \sigma_1$ then $x(t) = 0$ for all $t \in [t_k, t_{k+1}]$. We always have $y_i \ge 0$ so this implies $x(t) \le y_1(t) + y_2(t)$ for all $t \in [t_k, t_{k+1}]$.
\end{proof}

The following lemma will play the same role as Lemma~\ref{lem:Lambda_instantaneous_tumble}.

\begin{Lem}\label{lem:formula_for_Lambda_1} Set $\tau^{(i)}_L := \inf \{ s \ge 0 : y_i(s) \ge L \}$. There exists $U = U(\alpha, \beta, c, v) < 0$ such that for all $u \in (U, 0)$ and all $(y_i, \sigma_i) \in \mathbb R_+ \times \{1, 0, -1\}$ we have
	$$
	\lim_L \frac{1}{L} \log \mathbb E_{(y_i, \sigma_i)} \left[ e^{u \tau_L^{(i)}} \right] = - \max \left(\text{\normalfont spec}~\mathcal A_1(u)\right) =: \Lambda_1(u)
	$$
	where $\mathcal A_1(u) = - \mathcal V_1^{-1}\left(\mathcal Q_1 + uI\right)$ and
	$$
	\mathcal V_1 = \begin{pmatrix}
	-c + v & 0 & 0\\
	0 & -c & 0 \\
	0 & 0 & -c -v
	\end{pmatrix}, \quad
	\mathcal Q_1 = \begin{pmatrix}
	-\alpha & \alpha & 0 \\
	\beta/2 & -\beta & \beta/2 \\
	0 & \alpha& - \alpha
	\end{pmatrix}.
	$$
\end{Lem}

\begin{proof} If we can show that \( \mathcal A_1(u) \) has one positive as well as two distinct negative eigenvalues when \( u \to 0- \), then the result follows from the same arguments as in Lemma~\ref{lem:Lambda_instantaneous_tumble}. The eigenvalues of $A_1(0)$ are the roots of
	$$
	\det(\mathcal V)\det(X I - \mathcal A_1(0)) = X \left(c(v^2 - c^2) X^{2} - {\left(2  {\alpha} c^{2} + {\beta} c^{2} - {\beta} v^{2}\right)} X - {\alpha}^{2} c - {\alpha} {\beta} c  \right)
	$$
	so $A_1(0)$ has three distinct real eigenvalues $X_+ > 0, X_0 = 0, X_- < 0$. For $u \approx 0$, these can be extended in a continuously differentiable manner to three distinct eigenvalues $X_+(u), X_0(u), X_-(u)$ of $\mathcal A_1(u)$ using the implicit function theorem (the invertibility condition is guaranteed by the fact that the eigenvalues are distinct). Clearly $X_+(u) > 0$ and $X_-(u) < 0$ for $u$ small enough and the implicit function theorem also yields that the derivative of $X_0(u)$ at $u = 0$ is $\frac{1}{c}$. Hence $X_0(u) < 0$ when \( u \to 0- \).
\end{proof}

The last ingredient is the following lemma. Its proof is included for the sake of completeness.

\begin{Lem} \label{lem:union_bound_for_exponential_decrease}
Let $Y_1$ and $Y_2$ be independent a.s.~positive random variables. For all $A \ge 0$ and all $n \ge 1$ we have
$$
\mathbb P \left( Y_1 + Y_2 \ge A \right) \le \sum_{k = 0}^n \mathbb{P} \left( Y_1 \ge \frac{k}{n+1} A\right) \mathbb P\left(Y_2 \ge \frac{n - k}{n+1} A\right).
$$
\end{Lem}

\begin{proof} If we can show
$$
\{y_1 + y_2 \ge A \} \cap \{y_1, y_2 \ge 0\} \subset \bigcup_{k = 0}^n \left\{ y_1 \ge \frac{k}{n+1} A \right\} \cap \left\{ y_2 \ge \frac{n-k}{n+1}\right\}
$$
then we can conclude using the union bound and the independence of $Y_1$ and $Y_2$. \\

We have
$$
y_1 \ge \frac{k}{n+1} A \iff \frac{(n+1) y_1}{A} \ge k
$$
so setting $k^* := \left\lfloor \frac{(n+1) y_1}{A} \right\rfloor$ guarantees $y_1 \ge \frac{k^*}{n+1} A$. Furthermore, $y_1 + y_2 \ge A$ implies
\begin{align*}
\frac{n - k^*}{n+1} A = \frac{n - \left\lfloor \frac{(n+1) y_1}{A} \right\rfloor}{n+1} A \le \frac{n -  \frac{(n+1) y_1}{A} + 1}{n+1}A = A - y_1 \le y_2.
\end{align*}
Finally $k^* \ge 0$ and $k^* \le n$ except if $y_1 = A$ and $y_2 = 0$ but in this case we can take $k^* = n$.
\end{proof}

We can now adapt the computations of the previous section and obtain an upper bound for the exponential decay rate.

\begin{Lem}
\begin{itemize}
\item[(i)] For all $n \ge 1$, $u \in (U, 0)$ and $(x, \sigma) \in E$ we have
$$
\varlimsup_{t\rightarrow +\infty} \frac{1}{t} \log \mathbb P_{(x, \sigma)} \left( x(t) > a t\right) \le \frac{n}{n+1} a \Lambda_1(u) - 2u
$$
where $\Lambda_1(u)$ equals minus the largest eigenvalue of $\mathcal A_1(u)$.
\item[(ii)] For all $a > \frac{2}{\Lambda_1'(0)}$ we have
$$
\varliminf_{t \rightarrow +\infty} \frac{1}{t} \log \mathbb P_{(x, \sigma)} \left( x(t) > a t\right) > a\lambda_\pi.
$$
\item[(iii)] We have $\varlimsup_{t \rightarrow +\infty} -\frac{1}{t} \log \tv{\delta_{(\sigma, x)} P_t - \pi} \le 4 \alpha \left(1 + \frac{\alpha}{\beta}\right) \frac{c^2}{v^2}$.
\end{itemize}
\end{Lem}

\begin{proof} (i) Lemma~\ref{lem:union_bound_for_exponential_decrease} implies
\begin{align*}
\mathbb P \left(x(t) > a t\right) &\le \mathbb P\left( y_1(t) + y_2(t) \ge a t \right) \\
&= \sum_{k=0}^n \mathbb P \left( y_1(t) \ge \frac{k}{n + 1} a t\right) \mathbb P\left( y_2(t) \ge \frac{n - k}{n+1} a t\right) \\
&\le \sum_{k=0}^n \mathbb P \left( \tau^{(1)}_{ \frac{k}{n + 1} a t} \le t \right) \mathbb P\left( \tau^{(2)}_{ \frac{n - k}{n+1} a t} \le t \right) \\
&\le \sum_{k=0}^n e^{-2u t} \mathbb E \left[ e^{u \tau^{(1)}_{ \frac{k}{n+1} a t}}\right] \mathbb E \left[ e^{u \tau^{(2)}_{ \frac{n - k}{n+1} a t}}\right]
\end{align*}
and the result follows from Lemma~\ref{lem:formula_for_Lambda_1}. \\

(ii) Follows from the same arguments as Lemma~\ref{lem:instantaneous_linear_exponential_decay_rate_upper _bound}. \\

(iii) Assertion (ii) and Lemma~\ref{lem:finite_tumble_invariant_measure_exponential_decay_rate} imply that $\frac{2}{\Lambda_1'(0)}(-\zeta_2)$ is an upper bound for the exponential decay rate. Using the implicit value theorem to compute $\Lambda_1'(0)$ yields
\begin{align*}
\frac{2}{\Lambda_1'(0)}(-\zeta_2) &= 2\alpha \Bigg( \left(4 r^2+2 r-1\right) R^2+R^4 \left(\sqrt{\left(4 r^2-2\right) R^2+R^4+1}+4 r \left(2 r^2+r-1\right)-1\right) \\
&\quad \quad-\sqrt{\left(4 r^2-2\right) R^2+R^4+1}+(2 r+1) R^6+1 \Bigg) \\
&\quad \quad \times \left(2 r \left(\left(4 r^2-1\right) R^4+(4 r-1) R^2+R^6+1\right)\right)^{-1}
\end{align*}
where $r = \frac{\alpha}{\beta}$ and $R = \frac{c}{v}$. Optimizing over $r > 0$ and $0 < R < 1$ with the help of a computer algebra system we get
$$
\sup \frac{\frac{2}{\Lambda_1'(0)}(-\zeta_2)}{\alpha (1 + r) R^2} = 4
$$
which implies the desired result.
\end{proof}

\paragraph{Acknowledgments} The author warmly thanks Manon Michel and Arnaud Guillin for their kind feedback and advice.

\paragraph{Declaration of generative AI in the writing process} During the preparation of this work, the author used ChatGPT to improve spelling, grammar and style. After using this tool, the author reviewed and edited the content as needed and takes full responsibility for the content of the published article.

\bibliographystyle{alpha}
\bibliography{biblio.bib}

\end{document}